\documentclass[11pt]{amsart}
\usepackage{verbatim,amssymb,amsmath,mathrsfs,graphicx,hyperref}
\usepackage{xcolor}
  \scrollmode

\newcommand{\R}{{\mathbb R}}
\newcommand{\N}{{\mathbb N}}

\newcommand{\EE}{{\mathbb E}}
\newcommand{\PP}{{\mathbb P}}
\newcommand{\C}{{\mathbb C}}

\newcommand{\Lc}{\mathcal L}
\newcommand{\Wc}{\mathcal W}

\newcommand{\uti}{\underline t_i}

\newcommand{\eps}{\varepsilon}

\newcommand{\F}{\mathscr{F}}

\newcommand{\Fc}{\mathcal{F}}

\newcommand{\tX}{\widetilde X}
\newcommand{\tY}{\widetilde Y}
\newcommand{\tW}{\widetilde W}
\newcommand{\tB}{\widetilde B}
\newcommand{\tx}{\tilde x}
\newcommand{\bi}{\mathbf i}
\newcommand{\tpi}{\widetilde \pi}
\newcommand{\tPi}{\widetilde \Pi}
\newcommand{\pistar}{\pi^\ast}

\theoremstyle{plain}
\newtheorem{theorem}{Theorem}

\newtheorem{lemma}{Lemma}

\theoremstyle{definition}

\begin{document}

\title[
Approximation of SDEs with  a drift  coefficient of fractional Sobolev regularity]{
	On optimal error rates 
	for strong approximation of SDEs \\ with a drift coefficient \\ of fractional Sobolev regularity}

\author[Ellinger]
{Simon Ellinger}
\address{
	Faculty of Computer Science and Mathematics\\
	University of Passau\\
	Innstrasse 33 \\
	94032 Passau\\
	Germany} \email{simon.ellinger@uni-passau.de}

\author[M\"uller-Gronbach]
{Thomas M\"uller-Gronbach}
\address{
Faculty of Computer Science and Mathematics\\
University of Passau\\
Innstrasse 33 \\
94032 Passau\\
Germany} \email{thomas.mueller-gronbach@uni-passau.de}

\author[Yaroslavtseva]
{Larisa Yaroslavtseva}
\address{
	Institute of Mathematics and Scientific Computing\\
	University of Graz\\
	Heinrichstra{\ss}e 36 \\
	8010 Graz\\
	Austria} \email{larisa.yaroslavtseva@uni-graz.at}

\begin{abstract} 
We study strong approximation of scalar additive noise driven  stochastic differential equations (SDEs)
 at time point $1$
 in the case that  the drift coefficient is bounded and 
  has Sobolev regularity $s\in(0,1)$.
Recently, it has been shown in~\cite{DGL22} that for such SDEs the equidistant Euler approximation achieves an $L^2$-error rate of at least $(1+s)/2$, up to an arbitrary small $\varepsilon$,
in terms of the number of evaluations of the driving Brownian motion $W$.
In the present article  we  
prove a matching  lower error bound
 for $s\in(1/2,1)$.   More precisely we show that, for 
 every
 $s\in(1/2,1)$,
the $L^2$-error rate $(1+s)/2$ can,
up  to a logarithmic term,
not be improved in general by  no numerical 
 method based on finitely many evaluations of $W$ at fixed time points. Up to now, this result was known in the literature only for the cases $s=1/2-$ and $s=1-$.

 For the proof we employ the coupling of noise technique  recently introduced in~\cite{MGY23} to bound the $L^2$-error of an arbitrary approximation from below by the $L^2$-distance 
 of two occupation time functionals provided by a specifically chosen drift coefficient with Sobolev regularity $s$ and two  solutions of the corresponding SDE with coupled driving Brownian motions. For the analysis of the latter distance we employ a transformation of the original SDE to overcome the problem of correlated increments of the difference of the two coupled solutions, occupation time estimates to cope with the lack of regularity of the 
 chosen drift coefficient around 
 the point $0$ and scaling properties of the  drift coefficient.

\end{abstract}

\maketitle

\section{Introduction and main results}\label{s3}

Consider a scalar autonomous stochastic differential equation (SDE)
\begin{equation}\label{sde0}
	\begin{aligned}
		dX_t & = \mu(X_t) \, dt +  dW_t, \quad t\in [0,1],\\
		X_0 & = x_0
	\end{aligned}
\end{equation}
with deterministic initial value $x_0\in\R$, drift coefficient $\mu\colon\R\to\R$ and a one-dimensional driving
Brownian motion $W=(W_t)_{t\in[0,1]}$ and note that~\eqref{sde0} has a unique strong solution if $\mu$ is measurable and bounded, see~\cite{V80}.

For every $n\in\N$, let $X^\text{E}_n$ denote the corresponding continuous-time Euler scheme with $n$ equidistant steps, i.e. $X^\text{E}_{n,0} = x_0$ and
\[
X^\text{E}_{n,t} =X^\text{E}_{n,(i-1)/n}+ \mu\bigl( X^\text{E}_{n,(i-1)/n}\bigr)\,(t-(i-1)/n) + W_t -W_{(i-1)/n}, \quad t\in ((i-1)/n,i/n]
\]
for $i=1,\dots,n$.

Recently, in~\cite{DGL22} the performance of the Euler scheme $X^\text{E}_{n}$ was studied in the case when the drift coefficient $\mu$ 
has
fractional Sobolev regularity. To be more precise,
for $s\in (0,1)$ and $p\in [1,\infty)$ let
\[
W^{s,p} = \biggl\{ f\colon \R\to\R\,\Bigl| \, f \text{ is measurable and }\int_\R\int_\R \frac{|f(x)-f(y)|^p}{|x-y|^{1+sp}}\, dx\, dy <\infty \biggr\}
\]
be the
space of functions $f$ that have Sobolev regularity of order $s$ with integrability exponent $p$.
The following error estimate is the consequence of a more general, multi-dimensional result in ~\cite{DGL22}.

\begin{theorem}\label{thm1}
Let $s\in (0,1)$, $p\in [2,\infty)$ and assume that $\mu$ is bounded and  	$\mu \in W^{s,p}$.
Then for all $\eps\in (0,\infty)$ there exists $c\in (0,\infty)$ such that for all $n\in\N$,
\[
\EE\Bigl[\,\sup_{t\in [0,1]}| X_t-X^\text{E}_{n,t}|^p\Bigr]^{1/p} 
 \le \frac{c}{n^{(1+s)/2-\eps}}.
\]
\end{theorem}

In particular, 
under the assumptions of Theorem \ref{thm1},
the Euler approximation $X^\text{E}_{n,1}$ of $X_1$  achieves an 
$L^p$-error
rate of at least $(1+s)/2-$ in terms of the number $n$ of evaluations of the Brownian motion $W$, i.e.,
for all $\eps\in (0,\infty)$ there exists $c\in (0,\infty)$ such that for all $n\in\N$,
\begin{equation}\label{eul}
\EE\bigl[| X_1-X^\text{E}_{n,1}|^p\bigr]^{1/p} \le \frac{c}{n^{(1+s)/2-\eps}}.
\end{equation}

We add that for $p=2$ and a slightly different setting for the drift coefficient $\mu$,  the bound \eqref{eul} was essentially 
already known from~\cite{NS19}.
To be more precise, in the latter paper it is assumed that $\mu=a+b$, where $a$ is bounded and has bounded,
continuous
 derivatives up to order $2$ and 
$b$ is bounded, Lebesgue integrable and belongs to $ W^{s,2}$. Moreover, in contrast to \eqref{eul}, the Euler approximation in~\cite{NS19} is based on a non-equidistant time discretization  if $s\in(1/2, 1)$. 

The estimate~\eqref{eul}
naturally leads to the question whether the $L^p$-error rate  $(1+s)/2-$ 
is essentially 
optimal in
 the class of all 
 approximations of $X_1$
 based on finitely many evaluations of  $W$ or whether there exists a method from this class that achieves  a better $L^p$-error rate than   $(1+s)/2$  under the assumptions of Theorem \ref{thm1}, i.e. for  	$\mu \in W^{s,p}$ 
 bounded.

 To the best of our knowledge, the answer to this question 
 is so far known in the literature only in the case $s=1-$ for all $p\in[2, \infty)$ and in the case  $s=1/2-$ for $p=2$.

To be more precise, it is well-known that for 
SDEs \eqref{sde0} with a 
drift coefficient $\mu$ that has bounded, continuous derivatives up to order $3$ and satisfies $\mu'(x_0)\neq 0$,
the
best possible $L^p$-error rate that can be achieved by any numerical method based on finitely many evaluations of $W$ is at most $1$ for all $p\in[1, \infty)$, i.e. 
\[
\inf_{\substack{
		t_1,\dots ,t_n \in [0,1]\\
		g \colon \R^n \to \R \text{ measurable} \\
}}	 \EE\bigl[|X_1-g(W_{t_1}, \ldots, W_{t_n})|^p\bigr]^{1/p}\geq  \frac{c}{ n},
	\]
where $c\in (0,\infty)$ does not depend on $n$, see~\cite{m04}. This result was recently generalized in~\cite{HHMG18} to the case
that
the drift coefficient $\mu$ has the above regularity only locally, in a small neighborhood of the initial value $x_0$. In particular, 
the assumptions from~\cite{HHMG18} are satisfied for the SDE \eqref{sde0} with  $\mu=(1-|\cdot|)1_{[-1,1]}$ and $x_0\in(-1,0)\cup(0, 1)$. Since 
the latter choice of $\mu$ is Lipschitz continuous and has bounded support, it belongs
to the 
space $W^{s,p} $ for all $s\in(0,1)$ and all $p\in[1, \infty)$, see, e.g., \cite[Example 2.17]{Ern21} and \cite[Lemma 5.1]
{DNPV2012}, which demonstrates
that, for all $p\in[2, \infty)$,
the $L^p$-error rate  $(1+s)/2-$ in \eqref{eul} can essentially  not be improved in general  for $s=1-$. 

Moreover, recently in~\cite{ELL24,MGY23} it has been shown that for 
SDEs \eqref{sde0} with a 
drift coefficient $\mu$ that is piecewise differentiable with a bounded Lipschitz derivative and has at least one discontinuity,
the best possible $L^p$-error rate that can be achieved by any numerical method based on finitely many evaluations of $W$ is at most $3/4$ for all $p\in[1, \infty)$, i.e.
\begin{equation}\label{lbdisc}
\inf_{\substack{
		t_1,\dots ,t_n \in [0,1]\\
		g \colon \R^n \to \R \text{ measurable} \\
}}	 \EE\bigl[|X_1-g(W_{t_1}, \ldots, W_{t_n})|^p\bigr]^{1/p}\geq  \frac{c}{ n^{3/4}},
	\end{equation}
where $c\in (0,\infty)$ does not depend on $n$. This lower bound  
applies in particular 
in the case of $\mu=1_{[0,1 ]}$.
Observing that $1_{[0,1]}\in W^{s,p} $ for all $s\in(0, 1/p)$ and all $p\in[1, \infty)$, see, e.g.,
\cite[Section 3.1]{Sickel2021}
 we conclude that the $L^2$-error rate  $(1+s)/2-$ in \eqref{eul}  can 
 essentially  not be improved in general  for $s=1/2-$.

In the present article we show that  the $L^2$-error rate  $(1+s)/2-$ 
of the equidistant Euler approximation $X^\text{E}_{n,1}$ in \eqref{eul}  
is essentially optimal
for all $s\in (1/2,1)$. More formally, the main result of this article is the following theorem.

\begin{theorem}\label{thm2}
	For
	every
	$s\in (1/2,1)$
	there exists $c\in (0,\infty)$ and a bounded, 
	Lebesgue
	integrable  $\mu\in W^{s,2}$ 
such that  for all $n\in\N$,
	\begin{equation}\label{Mainlb}
\inf_{\substack{
		t_1,\dots ,t_n \in [0,1]\\
		g \colon \R^n \to \R \text{ measurable} \\
}}	 \EE\bigl[|X_1-g(W_{t_1}, \ldots, W_{t_n})|^2\bigr]^{1/2}\geq  \frac{c}{ \ln (n+1)n^{(1+s)/2}}.
	\end{equation}
\end{theorem}

Note that Theorem \ref{thm2} covers all approximations of $X_1$ that are based on evaluations of $W$ at fixed times $t_1, \ldots, t_n\in[0, 1]$ but does not apply to adaptive approximations, i.e. approximations that may choose both
the number and the location of evaluations of $W$ in a sequential
way, i.e. dependent on the previously observed evaluations of $W$. 
While it is known from~\cite{MG02_habil, m04}  that  for a large class of SDEs \eqref{sde0} with a
Lipschitz continuous drift coefficient $\mu$,  adaptive approximations can not achieve a better $L^p$-error rate than what is best possible for non-adaptive methods, the situation may drastically change if the drift coefficient is no longer Lipschitz continuous. In fact, recently in ~\cite{Y2022} an adaptive approximation was constructed that achieves  an $L^p$-error rate of at least  $1$ in terms of the average number of evaluations of $W$ under the assumption that the drift coefficient $\mu$ is piecewise differentiable with a bounded Lipschitz derivative, while in this case the best possible rate  for non-adaptive approximations is $3/4$ if $\mu$ has at least one discontinuity, see~\eqref{lbdisc}.
Whether Theorem \ref{thm2} extends to the class of  all   adaptive approximations is open so far and  a challenging question.

Theorem 2 provides a matching lower bound  for 
the upper bound
 \eqref{eul} in the case $s\in (1/2,1)$ and $p=2$. 
We conjecture that in fact  for all  $s\in (0,1)$ and all $p\in[2, \infty)$, the $L^p$-error rate  $(1+s)/2-$  in \eqref{eul}  can essentially not be improved
in general
for bounded $\mu \in W^{s,p}$. 
The proof
of this conjecture will be the subject of future work.

We add that numerical approximation of SDEs with 
a non-Lipschitz drift coefficient
is intensively studied since about $15$ years.
In particular, 
$L^p$-error rates have meanwhile been established for a wide range of non-Lipschitz conditions on  the drift coefficient, including local Lipschitz continuity, H\"older continuity, piecewise regularity and fractional Sobolev regularity, and also in different settings for the driving noise of the SDE. Moreover, SDEs with a non-Lipschitz diffusion coefficient are investigated as well.  See~\cite{MGY24} for 
a recent
survey on corresponding upper and lower $L^p$-error bounds and further references.
Most remarkably, in~\cite{DG18,DGL22} it was shown that the Euler approximation $X^\text{E}_{n,1}$ of $X_1$ achieves an $L^p$-error rate of at least $1/2-$ even if the drift coefficient $\mu$ is only bounded and measurable, and this result still holds for SDEs with non-additive noise under suitable assumptions on the diffusion coefficient.

We
finally
 would like to 
explicitly mention 
two further results on upper error bounds
 from the literature, which are strongly related to the upper bound \eqref{eul}.
In~\cite{MGY19b}, 
 a transformation-based Milstein-type scheme is constructed that achieves, for all $p\in [1,\infty)$, an $L^p$-error rate of at least  $3/4$ for SDEs with a discontinuous drift coefficient and multiplicative noise. 
More precisely, 
it is assumed in ~\cite{MGY19b} that
 the drift coefficient $\mu$ is piecewise Lipschitz continuous, the diffusion coefficient  is Lipschitz continuous and non-degenerate at the discontinuity points of $\mu$ and both coefficients have 
 piecewise a Lipschitz continuous derivative.
 Note that if, additionally, $\mu$ has a compact support then $\mu\in W^{s,2}$ for all $s\in(0, 1/2)$, see e.g. \cite[Section 3.1]{Sickel2021} and \cite[Theorem 4.6.4/2]{rs96} 
  and thus the upper bound \eqref{eul} yields an $L^2$-error rate of at least $3/4-$ for the approximation of the SDE \eqref{sde0}, 
 which is consistent with  the $L^2$-error rate of at least $3/4$ proven in ~\cite{MGY19b}. 
 Finally, in~\cite{butkovsky2020approximation, Ger23}, the $L^p$-approximation of SDEs with an $s$-H\"older continuous and bounded drift coefficient $\mu$ is studied 
for $s\in(0,1]$. In particular, it is shown in~\cite{butkovsky2020approximation} that, for all $p\in [1,\infty)$, the Euler approximation $X^\text{E}_{n,1}$ achieves an $L^p$-error rate of at least  $(1+s)/2-$ for the associated SDE \eqref{sde0} and in~\cite{Ger23} the same  $L^p$-error rate is established for the Milstein scheme for SDEs with a uniformly elliptic and $C^3_b$ diffusion coefficient. Again, if, additionally, $\mu$ has a compact support, then $\mu\in W^{u,2}$ for all $u\in(0, s)$, see e.g. 
\cite[Example 2.17]{Ern21} and \cite[Lemma 5.1]{DNPV2012},
and thus the upper bound \eqref{eul} yields an $L^p$-error rate of at least $(1+s)/2-$ for the associated SDE \eqref{sde0}, which is consistent with  the upper bounds  proven in ~\cite{butkovsky2020approximation,Ger23}.

The rest of the  article is organised as follows. We first present a sketch of our proof strategy of   Theorem \ref{thm2} in Section \ref{S1}.  In Section \ref{S2}, we introduce some notation and 
recall 
basic facts  about fractional Sobolev spaces
and the Fourier transformation,
which are relevant for the proof of Theorem \ref{thm2}.  In Section~\ref{S3} we provide essential technical tools for the  proof of Theorem \ref{thm2}, which is 
then 
carried out in Section \ref{proofthm}.

\section{On the proof of Theorem~\ref{thm2}}\label{S1}

The proof of Theorem \ref{thm2} is constructive. For $s\in (1/2,1)$, the  respective drift coefficient $\mu_s
\in W^{s,2}$
 is the Fourier transform 
\begin{equation}\label{mu1}
\mu_s = \F h_s
\end{equation}
of the function $h_s\colon \R\to\R$  given by 
\begin{equation}\label{mu2}
h_s(x)=\frac{1}{(e+|x|)^{1/2+s}\,\ln(e+|x|)}, \quad x\in\R.
\end{equation}

A slight modification of the function $\mu_s$
was recently used in~\cite{A2017} for establishing  a
 sharp
 lower error bound, 
 up to a logarithmic term,
 for the $L^2$-approximation of occupation time functionals
$
\int_0^1 \mu(W_t)dt
$
of the Brownian motion $W$ for integrands $\mu\in W^{s,2}$ 
based on 
equidistant 
evaluation of $W$.
More precisely, it was essentially shown in~\cite{A2017} that 
for all $n\in\N$,
\begin{equation}\label{altm}
\inf_{\substack{
		g \colon \R^n \to \R \text{ measurable} \\
}}	 \EE\Bigl[\Bigl|\int_0^1 \mu_s(W_t)dt-g(W_{1/n}, \ldots, W_{1})\Bigr|^2\Bigr]^{1/2}\geq  \frac{c}{ \ln (n+1) n^{(1+s)/2}},
	\end{equation}
where $c\in(0, \infty)$
does not depend on $n$.
 The proof of this result uses the fact that  the infimum in the left hand side of \eqref{altm} is attained by the conditional 
 expectation
 $\EE\bigl[\int_0^1 \mu_s(W_t)dt\bigl |W_{1/n}, \ldots, W_{1}\bigr]$,
whose 
mean squared error
is given by the sum
$
\sum_{i=1}^n\EE[Y_i^2]
$
with pairwise uncorrelated random variables
\begin{align*}
Y_i&=\int_{(i-1)/n}^{i/n} \bigl(\mu_s(W_t)-\EE\bigl[\mu_s(W_t) \bigl |W_{1/n}, \ldots, W_{1}\bigr]\bigr)dt
,\quad i=1,\dots,n,
\end{align*}
and the dependence of the expected values $\EE[Y_i^2]$ on $n$ can be 
analysed by using scaling properties of the Brownian motion $W$ and the Fourier 
transform
$\mathscr{F}$.

Our approximation problem is strongly related to the above approximation problem. Indeed, 
\[
X_1=x_0+\int_0^1 \mu_s(X_t)dt+W_1\qquad \PP\text{-a.s.}
\]
and hence the $L^2$-approximation of $X_1$ reduces to the $L^2$-approximation of the  occupation time functional
$
\int_0^1 \mu_s(X_t)dt
$
of the solution $X$ of \eqref{sde0}. However,     the direct analysis of the  
mean squared error
of the conditional 
expectation
$\EE\bigl[\int_0^1 \mu_s(X_t)dt\bigl |W_{t_1}, \ldots, W_{t_n}\bigr]$, similarly  to~\cite{A2017}, does not seem to be 
a feasible proof strategy
since the random variables
\[
Z_i=\int_{(i-1)/n}^{i/n} \bigl(\mu_s(X_t)-\EE\bigl[\mu_s(X_t) \bigl |W_{1/n}, \ldots, W_{1}\bigr]\bigr)dt,
\quad i=1,\dots,n,
\]
are 
probably
not uncorrelated in general 
and the expected values $\EE[Z_i^2]$
are  difficult to be analysed.

Instead,   
we employ the  coupling of noise approach, which was recently introduced in~\cite{MGY23} for proving the   lower error  bound \eqref{lbdisc} for SDEs with a 
drift coefficient $\mu$ that is piecewise Lipschitz continuous and has at least one discontinuity.
The idea of this approach is the following. Given $n\in\N$ and a discretization $\pi=\{t_1, \ldots, t_n\}$ of $[0,1]$ with $0<t_1<\ldots<t_n=1$  
we construct 
a second Brownian motion $\widetilde W^{\pi}$ such that  $W$ and $\widetilde W^{\pi}$ 
coincide
at the points $t_1,\dots,t_n$ but  are
independent, conditioned on $W_{t_1}, \ldots, W_{t_n}$. Let $\widetilde X^{\pi}=(\widetilde X^{\pi}_t)_{t\in[0,1]}$ denote the strong solution of the SDE \eqref{sde0} with driving Brownian motion  $\widetilde W^{\pi}$ in place of $W$.
Then 
\[
\inf_{g \colon \R^n \to \R \text{ measurable}} \EE\bigl[|X_1-g(W_{t_1},\dots,W_{t_n})|^2\bigr]^{1/2}\ge \frac{1}{2}\,\EE\bigl[|X_1-\widetilde X_1^{\pi}|^2\bigr]^{1/2}
\]
by the triangle inequality and the fact that $g(W_{t_1},\dots,W_{t_n}) = g(\widetilde W^{\pi}_{t_1},\dots,\widetilde W^{\pi}_{t_n})$,
see
Lemma \ref{lemf1}. In this way, the analysis of the smallest $L^2$-error for approximating $X_1$ based on $W_{t_1},\dots,W_{t_n}$ is reduced to the analysis of the $L^2$-distance between 
$X_1$ and $\widetilde X_1^{\pi}$.
We add that this idea is a special instance of what is called 'fooling algorithms' or 'diameter of information' in the field of 'Information-based Complexity': for a given numerical problem (here: $L^2$-approximation of $X_1$ based on evaluation of $W$ at the points $t_1,\dots,t_n$) we construct a second  instance of this problem (here: $L^2$-approximation of $\widetilde X^\pi_1$ based on evaluation of $\widetilde W^\pi$ at the points $t_1,\dots,t_n$) such that the admissible algorithms (here: the measurable functions $g\colon\R^n\to\R$) can not distinguish between the two problems based on the admissible information (here: evaluation of the driving Brownian motion at $t_1,\dots,t_n$) , see e.g.~\cite[Chapter 4]{NW08}.

For 
convenience
we restrict in the following to 
equidistant discretizations $\pi$, i.e. $t_i=i/n$ for $i=1,\dots,n$. 
For the analysis of the $L^2$-distance between $X_1$ and $\widetilde X_1^{\pi}$ we first apply a  bi-Lipschitz transformation $G\colon\R\to\R$ to $X$ and $\widetilde X^{\pi}$ that removes the drift coefficient such that the transformed stochastic processes $Y=(G(X_t))_{t\in[0,1]}$ and $\widetilde Y^{\pi}=(G(\widetilde X^{\pi}_t))_{t\in[0,1]}$ satisfy the SDE
\begin{equation}\label{sdeY}
	\begin{aligned}
		dY_t & =b(Y_t)dW_t, \quad t\in [0,1],\\
		Y_0 & = G(x_0)
	\end{aligned}
\end{equation}
and the SDE \eqref{sdeY} with $W$ replaced by $\widetilde W^{\pi}$, respectively, where $b\colon\R\to\R$ is globally Lipschitz continuous, see Lemma \ref{lem2}. By the Lipschitz continuity of $G$ we have
\[
\EE\bigl[|X_1-\widetilde X_1^{\pi}|^2\bigr]\geq c\cdot \EE\bigl[|Y_1-\widetilde Y_1^{\pi}|^2\bigr],
\]
where $c\in(0, \infty)$ does not depend on $\pi$.
The advantage of 
this
transformation procedure is that in contrast to the increments
\[
a_i
= (X_{t_i}- \tX^\pi_{t_i}) - (X_{t_{i-1}}- \tX^\pi_{t_{i-1}})=\int_{t_{i-1}}^{t_i} \bigl (\mu_s (X_{t}) -  \mu_s (\widetilde X_t^{\pi})  \bigr)\,
dt
\]
 of $X-\widetilde X^{\pi}$,
 the increments 
\[
		\alpha_i = (Y_{t_i}- \tY^\pi_{t_i}) - (Y_{t_{i-1}}- \tY^\pi_{t_{i-1}})=\int_{t_{i-1}}^{t_i}b(Y_t)dW_t-\int_{t_{i-1}}^{t_i}b(\widetilde Y^{\pi}_t)d\widetilde W^{\pi}_t
\]
of $Y-\widetilde Y^{\pi}$ over the time intervals $[t_{i-1}, t_i]$   are pairwise uncorrelated, and hence the representation
 \[
 \EE\bigl[|Y_1-\widetilde Y_1^{\pi}|^2\bigr]=\sum_{i=1}^
 {n}
 \EE[\alpha_i^2]
 \]
holds for the 
mean squared distance 
between $Y_1$ and $\widetilde Y_1^{\pi}$. 

The analysis of the increments $	\alpha_i$ 
is now
reduced to the analysis of the increments 
$a_i$
but with the processes $(X_t)_{t\in[t_{i-1},t_i]}$ and $(\widetilde X^\pi_t)_{t\in[t_{i-1},t_i]}$ replaced by the much simpler processes $(X_{t_{i-1}} + W_t-W_{t_{i-1}})_{t\in[t_{i-1},t_i]}$ and $(X_{t_{i-1}} +\widetilde W^\pi_t-  \widetilde W^\pi_{t_{i-1}})_{t\in[t_{i-1},t_i]}$, respectively. 
In fact, with 
\[
A_i
=\int_{t_{i-1}}^{t_i} \bigl (\mu_s (X_{t_{i-1} } +W_t - W_{t_{i-1} } ) -  \mu_s (X_{t_{i-1}} +\tW^\pi_t - \tW^\pi_{t_{i-1}})  \bigr)\,
dt
, \quad i=1,\dots,n,
\] 
we can show that
\begin{equation}\label{ll1}
	\sum_{i=1}^{n} \EE[\alpha_i^2]\geq c 
	\sum_{i=\lceil n/2 \rceil}^n
	\EE[A_i^2]-\frac{1}{n^{1+s+\varepsilon}}
\end{equation}
for some $c, \varepsilon\in(0, \infty)$
that do not depend on $n$,
see Lemma \ref{lemf3b}. For the proof of  \eqref{ll1} we employ smoothness  properties  of 
the transformation 
$G$, see Lemma \ref{lem1}, regularity properties of $\mu_s$,  see Lemma \ref{lem3}, 
upper bounds for the $L^4$-distance between $X_{t_{i-1}}$ and $\widetilde X^\pi_{t_{i-1}}$, which are obtained by applying Theorem \ref{thm1},
and
$L^2$-estimates
for the total time that $(X_s)_{s\in [t_{i-1},t_i]}$ and its approximation $(X_{t_{i-1}} + W_s-W_{t_{i-1}})_{s\in [t_{i-1},t_i]}$ spend on different sides of a horizontal line near $0$ in order
 to cope with the lack of regularity of the drift coefficient $\mu_s$ around $0$, see Lemma~\ref{Xlem1}(iv).

Finally, in Lemma~\ref{lemf5} we show 
that for $i\ge \lceil n/2 \rceil$, 
\begin{equation}\label{lx1}
\EE[A_i^2] \ge \frac{c}{\ln^2(n+1) n^{2+s}},
\end{equation}
where $c\in (0,\infty)$ does not depend on $n$, which completes the proof of Theorem~\ref{thm2}. 
The rough idea for the proof of~\eqref{lx1} is as follows. Using 
a Gaussian type lower bound for the density of $X_{t_{i-1}}$,
see  Lemma~\ref{Xlem1}(i), as well as 
  scaling properties of a Brownian motion
one obtains
 \begin{equation}\label{ll2}
\begin{aligned}
\EE[A_i^2] & \geq c_1 \int_{\R} e^{-c_2\,x^2}\EE\Bigl[\Bigl| \int_{t_{i-1}}^{t_i} \bigl (\mu_s (x +W_t - W_{t_{i-1} } ) -  \mu_s (x +\widetilde W^\pi_t - \widetilde W^\pi_{t_{i-1}})  \bigr)\, dt  \Bigr|^2\Big] \, dx\\
& 	= \frac{c_1}{n^{5/2}}  \int_\R e^{-c_2\,x^2/n}\,  \EE\Bigl[\Bigr|\int_0^1\bigl(\mu_s((x + W_t)/\sqrt{n} )- \mu_s((x + \widetilde W_t )/\sqrt{n} )\bigr)dt\Bigr|^2\Bigr] \, dx,
\end{aligned}
\end{equation}
where $c_1, c_2\in(0, \infty)$
do not depend on $n$ and $\widetilde W = \widetilde W^{\{1\}}$.
Using $\mu_s = \mathscr F h_s$ and 
scaling properties of the Fourier 
transform $\mathscr F$ we  may, asymptotically in $n$, replace the function $\mu_s(\cdot /\sqrt{n})$  by the function $(\ln(n+1))^{-1}(\sqrt n)^{-(s-1/2)} \mu_s$ in the right hand side of~\eqref{ll2}, which yields
\begin{align*}
\EE[A_i^2] & \ge \frac{c_1/2}{ \ln^2(n+1)n^{2+s}   }\int_\R e^{-c_2\,x^2/n}\, \eta(x) \, dx,
\end{align*} 
where
\[
\eta(x) = \EE\Bigl[\Bigr|\int_0^1\bigl(\mu_s(x + W_t) - \mu_s(x + \widetilde W_t ) \bigr)dt\Bigr|^2\Bigr].
\]
Finally, by Fatou's lemma, 
\[
\liminf_{n\to\infty} \int_\R e^{-c_2\,x^2/n}\, \eta(x) \, dx \ge \int_\R  \eta(x) \, dx > 0.
\]

\section{Notation}\label{S2}

We use $\bi$ to denote the imaginary unit in $\C$. For two measurable functions $f,g\colon\R\to\C$ we write $f=g\text{ a.e.}$ if $\lambda(\{f\neq g\}) = 0$.

For any $A\subset\R$ and any function $f\colon A\to\C$ we put $\|f\|_\infty = \sup_{x\in A} |f(x)|$.

For $k\in\N\cup\{\infty\}$ we use $C^k(\R)$ to denote the space of $k$ times continuously differentiable functions $f\colon \R\to\C$. The space of continuous  functions $f\colon \R\to\C$ is denoted by $C(\R)$ and we put $C_b(\R) = \{f\in C(\R)\colon f\text { is bounded}\}$.

For $\alpha\in [0,1)$ we use $C^\alpha(\R)$  to denote the space of functions $f\colon\R\to\C$ that are H\"older continuous of order $\alpha$, i.e. 
\[
[f]_{C^\alpha} = \sup_{x,y\in\R, \, x\neq y} \frac{|f(x)-f(y)|}{|x-y|^\alpha} < \infty.
\]

For $p\in [1,\infty)$ we use $\Lc^p(\R)$ to denote the space of measurable functions $f\colon\R\to\C$ that satisfy
\[
\|f\|_{\Lc^p} = \Bigl(\int_{\R} |f(x)|^p\, dx\Bigr)^{1/p}<\infty.
\]
For $s\in (0,1)$ and $p\in [1,\infty)$ we use $\Wc^{s,p}$ to denote the 
fractional Sobolev
 space of functions $f\in \Lc^p(\R)$ that satisfy
\[
[f]_{W^{s,p}} =\Bigl(\int_{\R} \int_{\R} \frac{|f(x)-f(y)|^p}{|x-y|^{1+sp}}\, dx\, dy \Bigr)^{1/p}< \infty
\]
and we put
\[
\Wc^{s-,p} = \bigcap_{s'\in (0,s)} \Wc^{s',p}.
\]
Note that for every $p\in [1,\infty)$ and all $ 0 < s' < s <1$,
\begin{equation}\label{z0}
	\Wc^{s,p} \subset \Wc^{s',p},
\end{equation}
see, e.g.~\cite[Proposition 2.1]{DNPV2012}.

For $f\in \Lc^1(\R) \cup \Lc^2(\R)$ we use $\F f\colon \R\to\C$ to denote a version of the Fourier transform of $f$. In particular, for $f\in \Lc^1(\R)$ we have
\[
\F f =  \int_\R e^{\bi z(\cdot)}\, f(z)\, dz\, \text{ a.e.}
\]

For later purposes we add that for every $f\in  \Lc^2(\R)$ and every $s\in (0,1)$ we have 
\begin{equation}\label{equiv}
	[f]_{W^{s,2}} =  c_s \int_\R |x|^{2s} |\F f(x)|^2\, dx,	
\end{equation}
where $c_s\in (0,\infty)$ only depends on $s$, see, e.g.,\cite[Proposition 3.4]{DNPV2012}.

For $f,g\colon\R\to \C$ we use 
\[
f\ast g\colon\R\to\C,\,\, x\mapsto \int_\R f(z)\, g(z-x)\, dz
\] 
to denote the convolution of $f$ and $g$, whenever this function is well-defined.

\section{Preliminaries}\label{S3}

In this section we provide essential tools that are employed in the proof of Theorem~\ref{thm2}, namely, the construction and properties of the bi-Lipschitz transformation used to transform the SDE~\eqref{sde0}, 
see Section~\ref{Trans}, estimates of the density and of certain occupation times of the solution of~\eqref{sde0}, see Section~\ref{solution}, and properties of the particular drift coefficient $\mu_s$ that is used in the proof of Theorem~\ref{thm2}.

\subsection{The transformation}\label{Trans}
The following transformation is a well-known tool to remove the drift coefficient of an SDE, see e.g.~\cite{NS19,Zvonkin74}.

Let $\mu\colon \R\to \R$ be locally integrable and define 
\begin{equation}\label{trans}
	G_\mu\colon \R\to\R, \, \, x\mapsto \int_0^x e^{-2\int_0^y \mu(z)\, dz}\, dy.
\end{equation}

We make use of the following well-known properties of $G_\mu$. 

\begin{lemma}\label{lem1}
	If  $\mu\in \Lc^1(\R)$ then $G_\mu$ has the following properties.\\[-.3cm]
	\begin{itemize}
		\item [(i)] $G_\mu\in C^1(\R)$ and there exist $c_1,c_2\in (0,\infty)$ such that $c_1 \le G_\mu'\le c_2$.\\[-.3cm]
		\item[(ii)] $G_\mu$ is a bijection with $G_\mu^{-1}\in C^1(\R)$ and $c_2^{-1} \le (G_\mu^{-1})'\le c_1^{-1}$.\\[-.3cm]
		\item [(iii)] $G_\mu'$ is absolutely continuous with weak derivative $-2\mu G_\mu'$.\\[-.3cm]
	\end{itemize}
	If, additionally, $\mu$ is bounded then $G_\mu'$ is Lipschitz continuous.
\end{lemma}

For convenience of the reader we provide a proof Lemma~\ref{lem1}.

\begin{proof} 
	Note that $\mu\in \Lc^1(\R)$ implies that the mapping 
	\[
	T\colon \R\to\R,\,\, x\mapsto  \int_0^x \mu(z)\, dz 
	\]
	is absolutely continuous. As a consequence, $G_\mu$ is continuously differentiable with 
	\[
	G_\mu'(x) = e^{-2T(x)},\,\, x\in \R.
	\] 
	Moreover, for every $x\in \R$,
	\[
	e^{-2\|\mu\|_{\Lc^1}} \le e^{-2T(x)} \le 	e^{2\|\mu\|_{\Lc^1}}, 
	\]
	which completes the proof of part (i). Part (ii) is an immediate consequence of part (i).
	
	Since $T$ is absolutely continuous with weak derivative $\mu$ and the mapping $S\colon\R\to \R$, $y\mapsto e^{-2y}$ is continuously differentiable, we obtain that $G_\mu'=S\circ T$ is absolutely continuous with weak derivative $(S'\circ T) \cdot \mu = -2G_\mu'\cdot \mu$. This completes the proof of part (iii). 
	
	The final statement in the lemma is a consequence of (i) and (iii) and the fact that $\|\mu G_\mu'\|_\infty \le  \|\mu\|_\infty \|G_\mu'\|_\infty$.
\end{proof}

Applying $G_\mu$ to the solution $X$ of the SDE~\eqref{sde0} yields the solution $Y$ of an SDE with drift coefficient zero and a Lipschitz diffusion coefficient. 

\begin{lemma}\label{lem2}
	Assume that $\mu$ is bounded and $\mu\in \mathcal L^1(\R)$. 
	Then 
	the SDE~\eqref{sde0} has a unique strong solution $X$ and the
	stochastic
	process                         
	\[
	Y=\bigl(Y_t = G_\mu(X_t)\bigr)_{t\in[0,1]}
	\]
	is 
	the  
	unique strong solution of the SDE 
	\begin{equation}\label{sde1}
		\begin{aligned}
			dY_t & = b_\mu(Y_t) \,   dW_t, \quad t\in [0,1],\\
			Y_0 & = G_\mu(x_0),
		\end{aligned}
	\end{equation}
	where $b_\mu = G_\mu'\circ G_\mu^{-1}$ is Lipschitz continuous.
\end{lemma}

For convenience of the reader we also provide a proof Lemma~\ref{lem2}.

\begin{proof}
	By assumption, $\mu$ is measurable and bounded, which implies existence and uniqueness of a strong solution of~\eqref{sde0}, see
	~\cite{V80}. By Lemma~\ref{lem1} we may apply a generalized It\^{o} formula, see e.g.~\cite[Problem 3.7.3]{ks91}, to conclude that the 
	stochastic
	process $(G_\mu(X_t))_{t\in[0,1]}$ is a strong solution of the SDE~\eqref{sde1}. By Lemma~\ref{lem1}, the functions $G_\mu'$ and $G_\mu^{-1}$ are Lipschitz continuous, which implies that $b_\mu$ is Lipschitz continuous. As a consequence, the strong solution of the SDE~\eqref{sde1} is unique.
\end{proof}

\subsection{Properties of  the solution}\label{solution}

We provide properties of the solution $X^x$ of the scalar 
SDE
\begin{equation}\label{sdegeneral}
	\begin{aligned}
		dX^x_t & = \mu(X^x_t)\, dt + dW_t,\,\, t\in[0,1],\\ X^x_0 & = x
	\end{aligned} 
\end{equation}
with initial value $x\in\R$ that are crucial for the proof of Theorem~\ref{thm2}.

\begin{lemma}\label{Xlem1}
	Assume that $\mu$ is measurable and bounded, let $(\Omega,\mathcal A,\PP)$ 
		be a 
	complete probability space, let $W\colon [0,1]\times \Omega\to \R$ be a standard Brownian motion, and for every $x\in\R$ let $X^x\colon  [0,1]\times \Omega\to \R$ be a strong solution of the SDE~\eqref{sdegeneral} with driving Brownian motion $W$ and initial value $x$. Then
	\begin{itemize}
		\item[(i)] there exist $c_1,c_2,c_3,c_4\in (0,\infty)$ such that for all $x\in\R$ and all $t\in (0,1]$, the distribution of $X^x_t$ has a Lebesgue density $f_{X^x_t}\colon \R\to[0,\infty)$, which satisfies for all $y\in\R$,
		\[
		c_1\frac{1}{\sqrt{2\pi c_2 t}} e^{-\frac{(y-x)^2}{2c_2 t}} \le	f_{X^x_t}(y) \le c_3\frac{1}{\sqrt{2\pi c_4t}} e^{-\frac{(y-x)^2}{2c_4 t}},
		\]
		\item [(ii)] for all $x\in\R$ and all $\tau\in (0,1]$ there exists $c\in (0,\infty)$ such that for all $t\in [\tau,1]$ and all $u,v\in\R$ with $u\le v$, 
		\[
		\PP(X^x_t\in [u,v] )\leq c\, (v-u),   
		\]
		\item[(iii)] there exists $c\in (0,\infty)$ such that for all $x,\xi\in\R$, all $s\in (0,1]$ and all $\varepsilon\in(0, \infty)$, 
		\[
		\int_{0}^{s}\PP(|X_t^x-\xi|\leq \varepsilon)\, dt \leq c\,\varepsilon\,\sqrt{s},
		\]
		\item[(iv)] there exists $c\in(0, \infty)$ such that for all $x,\xi\in\R$ and all $s,t\in [0,1]$ with $s<t$,
		\[
		\EE\Bigl[\Bigl(\int_{s}^{t}1_{\{(X^x_u-\xi)\, (X^x_s + W_u-W_s-\xi)\leq 0\}} \, du\Bigr)^2\Bigr]\leq c\,(t-s)^3.
		\]
	\end{itemize}
\end{lemma}

\begin{proof}
	Part (i) is a straightforward consequence of~\cite[Theorems 3.1, 3.2]{QZ2004}. Part (ii) follows  from the upper bound in part (i).
	Part (iii) can be proven in the same way as Lemma 4 in~\cite{MGY20}. For a detailed proof of part (iii) 
	see~\cite[Lemma 8]{MGY20arXiv}. Part (iv) is proven in the same way as Lemma 10 in~\cite{MGY23}.
\end{proof}

\subsection{The function $\mu_s$}\label{mumu}

For every $s\in (1/2,1)$ we define
\[
h_s\colon \R\to\R,\,\, x\mapsto \frac{1}{(e+|x|)^{1/2+s}\,\ln(e+|x|)}.
\]
Clearly, $h_s$ is bounded, even and for every $p\in [1,\infty)$ we have $h_s\in \Lc^p (\R)$. We define

\begin{equation}\label{mu}
	\mu_s \colon\R\to\R,\,\, x\mapsto 2\int_0^\infty \cos(xz)\, h_s(z)\, dz.
\end{equation}

Since $h_s$ is even and  $h_s\in \Lc^1(\R)$ we have 
\[
\mu_s = \F h_s\text{ a.e.}
\]
and $\mu_s$ is continuous and bounded.
We provide further properties of $\mu_s$ that are crucial for the proof of Theorem~\ref{thm2}.

\begin{lemma}\label{lem3} For every $s\in (1/2,1)$ we have\\[-.4cm]
	\begin{itemize}
		\item [(i)] $\forall x\in\R\backslash\{0\}\colon\, |\mu_s(x)| \le  \frac{4(3/2+s)}{x^2}$,\\[-.3cm]
		\item[(ii)] $\forall p\in [1,\infty)\colon\, \mu_s\in \Lc^p(\R)$,\\[-.3cm]
		\item[(iii)] $\mu_s\in \Wc^{s,2}$,\\[-.3cm]
		\item[(iv)]  $\forall p\in (2,\infty) \colon\, \mu_s\in \Wc^{s_p-,p}$, where $s_p = s-(1/2-1/p)$,\\[-.3cm]
		\item[(v)] $\exists c\in (0,\infty) \,\forall x,y\in \R\backslash\{0\}\colon\, |\mu_s(x)-\mu_s(y)| \le c \frac{|x-y|}{|x|}$,\\[-.3cm]    
		\item[(vi)] $\forall s'\in (1/2,s)\colon\, \mu_s\in C^{s'-1/2}(\R).$
	\end{itemize}
\end{lemma}

\begin{proof}
	Fix $s\in (1/2,1)$. First, note that $h_s$ is infinitely often differentiable on $(0,\infty)$ and  for every $x\in (0,\infty)$,
	\begin{equation*}
		\begin{aligned}
			h'_s(x) & = \frac{-(1/2+s)}{(e+x)^{3/2+s}} \, \frac{1}{\ln(e+x)} + \frac{1}{(e+x)^{1/2+s}} \, \frac{-1}{(e+x)\ln^2(e+x) }\\
			& = \frac{-1}{(e+x)^{3/2+s}\ln(e+x)} \,\Bigl( 1/2+s + \frac{1}{\ln (e+x)}\Bigr).
		\end{aligned}
	\end{equation*}
	In particular, for every $x\in (0,\infty)$,
	\begin{equation}\label{h0}
		| h_s'(x) |  = -h'(x) \le  \frac{3/2+s}{(e+x)^{3/2+s}\ln(e+x)}, 
	\end{equation}	
	and
	\begin{equation}\label{h1}
		h'_s(0+) = 	\frac{-(3/2+s)}{e^{3/2+s}} \quad\text{ and }\quad	\lim_{x\to\infty} h_s'(x) =\lim_{x\to\infty} h_s(x) =0.
	\end{equation}	
	Moreover, $h_s'$ is strictly increasing on $(0,\infty)$, which implies 
	\begin{equation}\label{h2}
		h_s'' > 0\text{ on }(0,\infty).
	\end{equation}	
	
	For the proof of part (i), let $x\in \R\backslash\{0\}$ and use partial integration as well as~\eqref{h1} to obtain 
	\begin{equation}\label{h3}
		\begin{aligned}
			\int_0^\infty \cos(xz)\, h_s(z)\, dz & = \frac{\sin(xz)}{x} \, h_s(z) \Bigr|^{z=\infty}_{z=0} - \int_0^\infty \frac{\sin(xz)}{x}\, h_s'(z)\, dz\\
			& =  \frac{\cos(xz)}{x^2} \, h_s'(z) \Bigr|^{z=\infty}_{z=0} - \int_0^\infty \frac{\cos(xz)}{x^2} \, h_s''(z)\, dz\\
			& = -\frac{h_s'(0+)}{x^2} - \int_0^\infty \frac{\cos(xz)}{x^2} \, h_s''(z)\, dz.
		\end{aligned}
	\end{equation}	
	Using~\eqref{h1} to ~\eqref{h3} we conclude that
	\[
	|\mu_s(x)|  \le \frac{2|h_s'(0+)|}{x^2} + \frac{2}{x^2} \int_0^\infty h_s''(z)\, dz = \frac{2}{x^2} \bigl(|h_s'(0+)| - h_s'(0+)   \bigr)\leq \frac{4(3/2+s)}{x^2},
	\]
	which completes the proof of part (i).

	For the proof of part (ii), let $p\in [1,\infty)$.
	Since $\mu_s$ is bounded we have
	\[
	\int_{-1}^1 |\mu_s(x)|^p\, dx <\infty.
	\]
	Using part (i) we obtain
	\begin{align*}
		\int_{\R\backslash [-1,1]} |\mu_s(x)|^p\, dx & \le 2(4(3/2+s))^p \int_1^\infty \frac{1}{x^{2p}}\, dx <\infty.
	\end{align*}
	Hence $\mu_s\in \Lc^p(\R)$. 
	
	We turn to the proof of part (iii). In view of part (ii) it suffices to show that
	\begin{equation}\label{z01}
		[\mu_s]_{W^{s,2}} < \infty.
	\end{equation}
	We have
	\begin{equation}\label{zzz1}
		\int_\R |x|^{2s}\, |h_s(x)|^2\, dx  = 2\int_0^\infty \frac{x^{2s}}{(e+x)^{2s+1}\,\ln^2(e+x)}\, dx
		\le 2\int_e^\infty \frac{1}{x\,\ln^2(x)} \, dx < \infty.
	\end{equation}
	Now, recall that $\mu_s = \F h_s$ 
	a.e.
	and $h_s\in \Lc^1(\R)\cap \Lc^2(\R)$ and 
	use~\eqref{equiv} as well as the Fourier inversion formula
	to obtain~\eqref{z01} from~\eqref{zzz1}.

	Next, we prove part (iv). Let $p\in (2,\infty)$. 
	In view of part (ii) and~\eqref{z0} it suffices to show that for every $\eps\in (0,s_p)$,  
	\begin{equation}\label{zz1}
		[\mu_s]_{W^{s_p-\eps/p,p}} < \infty.
	\end{equation}
	Put $t= s_p - \eps/p$. We have
	\begin{equation}\label{zz2}
		\begin{aligned}
			[\mu_s]_{W^{t,p}} & =\biggl( \int_\R \int_\R \frac{|\mu_s(x)-\mu_s(x+y)|^p}{|y|^{1+tp}}\, dx\, dy\biggr)^{1/p}\\
			& \le \biggl(\int_{ [-1,1 ] }  \frac{1}{|y|^{1+tp}}\int_\R |\mu_s(x)-\mu_s(x+y)|^p\, dx  \, dy \biggr)^{1/p} + 2^{1+1/p} (tp)^{-1/p}\, \|\mu_s\|_{\Lc^p}.
		\end{aligned}
	\end{equation}
	Fix $y\in\R$, recall that $h_s\in  \Lc^1(\R)\cap\Lc^2(\R)$ and note that  by part (ii) we have $\mu_s, \mu_s(\cdot + y)\in \Lc^1(\R)\cap\Lc^2(\R)$ as well. Employing the Hausdorff-Young inequality, the Fourier inversion formula and the fact that
	$|1-e^{-\bi u}| = 2|\sin(u/2)|$ and $|\sin(u)| \le |u|^\alpha$ for every $u\in \R$ and every $\alpha\in [0,1]$ we obtain
	\begin{equation}\label{zz3}
		\begin{aligned}
			\int_\R |\mu_s(x)-\mu_s(x+y)|^p\, dx 
			& \le (2\pi)^{-(p-1)}  \biggl(\int_\R |(\F\mu_s- \F\mu_s(\cdot +y))(x)|^{p/(p-1)}\, dx\biggr)^{p-1}\\	
			& = (2\pi)^{-(p-1)}  \biggl(\int_\R |(1-e^{-\bi xy})\F\mu_s(x)|^{p/(p-1)}\, dx\biggr)^{p-1}\\
			& =  (2\pi)^{-(p-1)}\biggl(\int_\R |2\sin(xy/2)|^{p/(p-1)} |h_s(x)|^{p/(p-1)}\, dx\biggr)^{p-1}\\  
			& \le  \biggl(\int_\R 2^{p/(p-1)} |xy/2|^{s_pp/(p-1)} |h_s(x)|^{p/(p-1)}\, dx\biggr)^{p-1}\\
			& = \frac{2^p|y|^{s_pp} }{2^{s_pp} } \biggl(\int_\R \frac{|x|^{s_pp/(p-1)}}{((e+|x|)^{1/2 +s}\ln(e+|x|))^{p/(p-1)}}\, dx\biggr)^{p-1}\\
			& \le \frac{2^p|y|^{s_pp} }{2^{s_pp} } \biggl(\int_\R \frac{1}{(e+|x|)(\ln(e+|x|))^{p/(p-1)} }\, dx\biggr)^{p-1}\\
			& \le c \, |y|^{s_pp},
		\end{aligned}
	\end{equation}
	where $c\in (0,\infty)$ only depends on  $p$.
	
	By~\eqref{zz3} we conclude that
	\begin{equation}\label{z4}
		\begin{aligned}
			\int_{ [-1,1 ] }  \frac{1}{|y|^{1+tp}}\int_\R |\mu_s(x)-\mu_s(x+y)|^p\, dx  \, dy & \le 
			c	\int_{ [-1,1 ] } 	 \frac{1}{|y|^{1-\eps}}\, dy <\infty.
		\end{aligned}
	\end{equation}	
	Combining~\eqref{zz2} with~\eqref{z4} and employing part (ii) yields ~\eqref{zz1}. 	
	
	We continue with the proof of part (v). Let $x,y\in\R\backslash \{0\}$. By \eqref{h3} we have
	\begin{equation}\label{zz5}
		\mu_s(x)-\mu_s(y) = - 2\int_0^\infty \Bigl(\frac{\sin(xz)}{x} - \frac{\sin(yz)}{y}\Bigr)\, h_s'(z)\, dz.
	\end{equation}
	For every $z\in\R$ we get
	\begin{equation}\label{zz6}
		\begin{aligned}
			\Bigl|\frac{\sin(xz)}{x} - \frac{\sin(yz)}{y}\Bigr| &  \le \Bigl|\frac{\sin(xz)-\sin(yz)}{x}\Bigr|  +   |\sin(yz)| \Bigl|\frac{1}{x} - \frac{1}{y}\Bigr|\\
			& \le \frac{|xz-yz|}{|x|} + |yz| \frac{|x-y|}{|xy|} =  \frac{2|x-y|}{|x|}\, |z|.
		\end{aligned}
	\end{equation}
	Combining~\eqref{zz5} with~\eqref{zz6} and employing~\eqref{h0} yields
	\begin{align*}
		|	\mu_s(x)-\mu_s(y)| & \le  4 \,  \frac{|x-y|}{|x|} \int_0^\infty z |h'(z)|\, dz \\
		& \le 4(3/2+s)  \, \frac{|x-y|}{|x|} \int_0^\infty \frac{1}{(e+z)^{1/2+s} \ln(e+z)}\, dz \le c\, \frac{|x-y|}{|x|},
	\end{align*}
	where $c\in (0,\infty)$ only depends on $s$.
	
	Finally, we prove part (vi). Let $s'\in (1/2,s)$ and let $x,y\in\R$. Clearly, for every $z\in(0,\infty)$,
	\begin{align*}
		|\cos(xz) - \cos(yz)| & \le 2^{3/2-s'} \, |\cos(xz) - \cos(yz)|^{s'-1/2} \\ &  \le 2^{3/2-s'} \, |zx-zy|^{s'-1/2} \le 2^{3/2-s'} |x-y|^{s'-1/2}\, |e+z|^{s'-1/2},
	\end{align*}
	and therefore,
	\begin{align*}
		|	\mu_s(x)-\mu_s(y)| & \le 2 \int_0^\infty |\cos(xz) - \cos(yz)| h_s(x)\, dz \\
		& \le 	2^{5/2-s'} \,|x-y|^{s'-1/2} \,\int_0^\infty  \frac{1}{(e+z)^{1+(s-s')} \, \ln(e+z)}\, dz \le c\, |x-y|^{s'-1/2},
	\end{align*}
	where $c\in (0,\infty)$ only depends on $s,s'$.
\end{proof}

\section{Proof of Theorem~\ref{thm2}}\label{proofthm}

In the following let 
$(\Omega,\mathcal A,\PP)$
be a probability space, let $W\colon [0,1]\times \Omega\to \R$ be a standard Brownian motion on $[0,1]$, 
let $x_0\in\R$,
let $\mu\colon \R\to\R$ be measurable and bounded  and let $X$ denote the strong solution of the corresponding SDE~\eqref{sde0}.

For a discretization $\pi =\{t_1,\dots,t_n\}$ of $[0,1]$ with $0 < t_1 <\dots <t_n = 1$  we use 
\[
e_2(\pi) = \inf_{        g \colon \R^n \to \R \text{ measurable}} \EE\bigl[  |X_1-g(W_{t_1},\dots,W_{t_n} )|^2\bigr]^{1/2}
\]
to denote the smallest possible 
$L^2$-error 
that can be achieved for approximating  $X_1$ based on  $W_{t_1},\dots,W_{t_n}$. To obtain a lower bound for $e_2(\pi)$ we construct, similar to~\cite{MGY23}, a Brownian motion $\tW^\pi\colon [0,1]\times \Omega\to\R$ that is coupled with $W$ at the points $t_1,\dots,t_n$ and we analyse the  
$L^2$-distance of the solution  $\tX^\pi$ of the 
SDE~\eqref{sde0} with driving Brownian motion $\tW^\pi$ and $X$ at the final time, see Lemma~\ref{lemf1}.

Put $t_0 =0$, let $\overline W^\pi\colon [0,1]\times \Omega\to\R$ denote the piecewise linear interpolation of $W$ on $[0,1]$ at the points $t_0, \ldots, t_{n}$, i.e. 
\[
\overline W^\pi_t=\tfrac{t-t_{i-1}}{t_i-t_{i-1}}\,W_{t_i}+\tfrac{t_i-t}{t_i-t_{i-1}}\, W_{t_{i-1}}, \quad t\in [t_{i-1}, t_i],
\]
for $i\in\{1, \ldots, n\}$, and  put
\[
B^\pi=W-\overline W^\pi.
\]
It is well known that $(B^\pi_t)_{t\in [t_{i-1}, t_i]}$ is a Brownian bridge on $[t_{i-1}, t_i]$ for every $i\in\{1, \ldots, n\}$ and that the stochastic processes
$(B^\pi_t)_{t\in [t_{0}, t_1]}, \ldots,(B^\pi_t)_{t\in [t_{n-1}, t_{n}]}, \overline W^\pi $  are independent. 
Without loss of generality, we may assume that $(\Omega,\mathcal A,\PP)$ 
is rich enough to carry 
for every $i\in\{1, \ldots, n\}$ a Brownian bridge $(\widetilde B^\pi_t)_{t\in [t_{i-1}, t_i]}$   on $[t_{i-1}, t_i]$
such that $(\widetilde B^\pi_t)_{t\in [t_{0}, t_1]}, \ldots,(\widetilde B^\pi_t)_{t\in [t_{n-1}, t_{n}]}, W $ are independent. Put $\widetilde B^\pi=(\widetilde B^\pi_t)_{t\in[0,1]}$ and
define a Brownian motion  $\widetilde W^\pi\colon [0,1]\times \Omega\to\R$  by
\[
\widetilde W^\pi= \overline W^\pi+\widetilde B^\pi.
\]
We use 
\begin{equation}\label{extra1}
	\tX^\pi = (\tX_t^\pi )_{t\in[0,1]}
\end{equation}
to denote the strong solution of the SDE~\eqref{sde0} with driving Brownian motion $\tW^\pi$ in place of $W$
and 
\[
\Fc^{W,\tW^\pi} = \bigl(\Fc^{W,\tW^\pi}_t = \sigma(\{(W_s,\tW^\pi_s)\colon s\in [0,t]\})\bigr)_{t\in [0,1]}
\]
to denote the filtration generated by the process $(W,\tW^\pi)$.  
Clearly, 
\begin{equation}\label{qx1}
	\forall i\in\{1,\dots,n\}\colon \, W_{t_i} = \tW^\pi_{t_i}
\end{equation}
and it is easy to check that 
\begin{equation}\label{qx2}
	\forall i\in\{1,\dots,n\}\colon \, \Fc^{W,\tW^\pi}_{t_i} \text{ and } \sigma\bigl(\{(W_t-W_{t_i},\tW^\pi_t-\tW^\pi_{t_i})\colon t\in[t_i,1]\}\bigr)\text{ are independent}.
\end{equation}

For all $n\in\N$ we put
\[
\Pi^n = \bigl\{\{t_1,\dots,t_n\}\colon 0<t_1< \dots <t_n=1\bigr\}
\]
and we define
\[
\Pi = \bigcup_{n\in\N} \Pi^n.
\]

See~\cite[Lemma 11]{MGY23} for a proof of the following result.

\begin{lemma}\label{lemf1} Let $\mu$ be measurable and bounded. Then for 
	every $\pi\in\Pi$, 
	\[
	e_2(\pi) \geq \frac{1}{2}\,   \EE[|X_1-\widetilde X^\pi_1|^2]^{1/2}.
	\]
\end{lemma}

Next, we use the transformation $G_\mu$, see~\eqref{trans}, to switch from a solution of the SDE~\eqref{sde0} to a  solution of the SDE~\eqref{sde1}, see Lemma~\ref{lem2}. Let
\begin{equation}\label{extra2}
	Y=(G_\mu(X_t))_{t\in[0,1]}
\end{equation}
and for every $\pi\in \Pi$ define
\begin{equation}\label{extra3}
	\widetilde Y^\pi=(G_\mu(\widetilde X^\pi_t))_{t\in[0,1]}.
\end{equation}

\begin{lemma}\label{lemf2} Let $\mu$ be bounded and satisfy $\mu\in\Lc^1(\R)$. Then there exists $c\in (0,\infty)$ such that for  
	all $\pi \in \Pi$,
	\begin{equation}\label{gL2}
		\EE\bigl[|X_1-\widetilde X^\pi_1|^2\bigr]^{1/2} \geq c\,   \EE[|Y_1-\widetilde Y^\pi_1|^2]^{1/2}.
	\end{equation}
\end{lemma}

\begin{proof}
	By Lemma~\ref{lem2}(i), the transformation $G_\mu$ is Lipschitz continuous, which obviously implies the claimed estimate.
\end{proof}

Next, we turn to the analysis of the right hand side in~\eqref{gL2} in case that $\mu$ has sufficient Sobolev regularity. To this end we restrict to discretizations $\pi$ that behave sufficiently well.

For every $n\in\N$ we define 
\[
\tPi^n = \bigl\{   \{t_1,\dots, t_{5n}\}\colon 0<t_1 <\dots < t_{5n}=1, \{j/(4n)\colon j\in\{1,\dots, 4n\}\} \subset \{t_1,\dots,t_{5n} \}\bigr\}.
\]
Clearly, every $\pi\in \tPi^n$ satisfies
\begin{equation}\label{ndisc1}
	\max_{i\in \{1,\dots,5n\}} (t_i -t_{i-1}) \le 1/(4n).
\end{equation}
Moreover, it is easy to check that 
\begin{equation}\label{ndisc2}
	\forall \pi \in \Pi^n \, \exists \tpi\in \tPi^n\colon \, \pi \subset \tpi	
\end{equation}
and 
\begin{equation}\label{ndisc3}
	\forall \pi \in \tPi^n\colon \, \#\bigl\{ i\in \{2,\dots, 5n\}\colon t_{i-1}\ge 1/2\text{ and } t_i-t_{i-1} = 1/(4n)\} \ge n.
\end{equation}

\begin{lemma}\label{lemf3} Let $s\in (1/2,1)$ and assume that $\mu\colon\R\to\R$ is bounded and satisfies
	\[
	\mu\in\Lc^1(\R)
	\cap \Wc^{(s-1/4)-,4}.
	\]
	Then for every $c_0\in (0,s-1/2)$ there exist  $c_1,c_2\in (0,\infty)$ such that for all $n\in\N$, all $\pi=\{t_1,\dots,t_{5n}\}\in\tPi^n$ with $0<t_1<\dots <t_{5n}=1$ and all $ i\in \{1,\dots,5n\}$, 
	\begin{equation}\label{iter1}
		\begin{aligned}
			\EE\bigl[|Y_{t_i} - \widetilde Y^\pi_{t_i}|^2\bigr] & \ge \left(1-\frac{c_1}{n}\right)\, \EE\bigl[|Y_{t_{i-1}} - \widetilde Y^\pi_{t_{i-1}}|^2\bigr] \\
			& \qquad\qquad + c_2\, \EE\bigl[|(X_{t_i}- X_{t_{i-1}}) - (\tX^\pi_{t_i}- \tX^\pi_{t_{i-1}})|^2\bigr] - \frac{c_1}{n^{2+s+c_0}}.
		\end{aligned}
	\end{equation}
\end{lemma}

\begin{proof}
	Fix $c_0\in (0,s-1/2)$, let $n\in\N$ and  $\pi=\{t_1,\dots,t_{5n}\}\in\tPi^n$ with $0<t_1<\dots <t_{5n}=1$ and
	put 
	\[
	\Delta_i  = \EE\bigl[|Y_{t_i} - \widetilde Y^\pi_{t_i}|^2\bigr]^{1/2}
	\]
	for  $i\in\{0,\dots,5n\}$ and 
	\[
	\alpha_i = (Y_{t_i}- \tY^\pi_{t_i}) - (Y_{t_{i-1}}- \tY^\pi_{t_{i-1}})
	\]
	for  $i\in\{1,\dots,5n\}$.  
	
			Fix $i\in\{1,\dots,5n\}$. Throughout the following we use $c\in (0,\infty)$ to denote a positive constant that does not depend on $n$ or $\pi$ or $i$. The value of $c$ may change from 
	line to line.
	By Lemma~\ref{lem1}(i) and Lemma~\ref{lem2}, the stochastic process $(b_\mu(Y_t))_{t\in[0,1]}$ is  bounded, measurable and adapted to 
	$\Fc=\bigl(\Fc_t = \sigma(\{W_s\colon s\in [0,t]\}\cup \mathcal N(\PP))\bigr)_{t\in [0,1]}$, where $\mathcal N(\PP)$ is the collection of $\PP$-null sets.
	Hence, 	for all $r,t\in[0,1]$,
	\begin{equation}\label{rstuv1}
		\EE[\tY_r^\pi \tY^\pi_t]	=	\EE[Y_r Y_t]  = \EE\biggl[\int_0^{\min(r,t)} b_\mu^2(Y_u)\, du\biggr] + G^2_\mu(x_0).
	\end{equation}	
     By the Lipschitz continuity of $b_\mu$,  strong existence and pathwise uniqueness hold for the
   SDE \eqref{sdeY}. Hence, there exists a measurable function $g\colon (C([0,t_{i-1}];\R))^2\to \R$  such that $\PP$-almost surely, 
	\[
	\tY_{t_{i-1}}^\pi = g((W_u)_{u\in [0,t_{i-1}]}, (\widetilde B^\pi_u)_{u\in [0,t_{i-1}]}),
	\]
	see, e.g., \cite[Theorem 1]{Ka96}.
	 Since $W$ and $\widetilde B^\pi$ are independent and $(W_u)_{u\in [0,t_{i-1}]}$ is measurable with respect to $\Fc_{t_{i-1}}$,
	we conclude that
	\begin{equation}\label{rstuv2}
		\begin{aligned}
			& \EE[\tY_{t_{i-1}}^\pi  (	Y_{t_i}-Y_{t_{i-1}})] \\
			& \qquad\qquad = \EE\Bigl[g\bigl((W_u)_{u\in [0,t_{i-1}]}, (\widetilde B^\pi_u)_{u\in [0,t_{i-1}]}\bigr) \int_{t_{i-1}}^{t_i} b_\mu(Y_s)\, dW_s\Bigr]\\
			& \qquad\qquad = \int_{C([0,t_{i-1}];\R)} \EE\Bigl[g\bigl((W_u)_{u\in [0,t_{i-1}]}, f\bigr) \int_{t_{i-1}}^{t_i} b_\mu(Y_s)\, dW_s\Bigr]\, \PP^{(\widetilde B^\pi_u)_{u\in [0,t_{i-1}]} }(df)\\
			& \qquad\qquad = \int_{C([0,t_{i-1}];\R)} \EE\Bigl[ \int_{t_{i-1}}^{t_i} g\bigl((W_u)_{u\in [0,t_{i-1}]}, f\bigr) b_\mu(Y_s)\, dW_s\Bigr]\, \PP^{(\widetilde B^\pi_u)_{u\in [0,t_{i-1}]} }(df).
		\end{aligned} 
	\end{equation}
	By boundedness of $b_\mu$ we have
	\[
	\EE\bigl[\sup_{s\in [t_{i-1},t_i]}| \tY_{t_{i-1}}^\pi b_\mu(Y_s)|^2 \bigr] \le \|b_\mu\|_\infty \EE \bigl[|\tY_{t_{i-1}}^\pi|^2\bigr] <\infty,
		\]
	which implies that for $\PP^{(\widetilde B^\pi_u)_{u\in [0,t_{i-1}]} }$-almost all $f\in C([0,t_{i-1}];\R)$,
	\[
	\EE\bigl[\sup_{s\in [t_{i-1},t_i]}|g\bigl((W_u)_{u\in [0,t_{i-1}]}, f\bigr) b_\mu(Y_s)|^2 \bigr]  <\infty.
	\]
	Hence, for $\PP^{(\widetilde B^\pi_u)_{u\in [0,t_{i-1}]} }$-almost all $f\in C([0,t_{i-1}];\R)$,
	\[
	\EE\Bigl[ \int_{t_{i-1}}^{t_i} g\bigl((W_u)_{u\in [0,t_{i-1}]}, f\bigr) b_\mu(Y_s)\, dW_s\Bigr] =0,
	\]
	which jointly with~\eqref{rstuv2} yields
	\begin{equation}\label{rstuv3}
		\EE[\tY_{t_{i-1}}^\pi  (	Y_{t_i}-Y_{t_{i-1}})] =0.
	\end{equation}
	For reasons of symmetry, we also have 
	\begin{equation}\label{rstuv4}
		\EE[Y_{t_{i-1}}  (	\tY^\pi_{t_i}-\tY^\pi_{t_{i-1}})] =0.
	\end{equation}

		By~\eqref{rstuv1}, ~\eqref{rstuv3} and ~\eqref{rstuv4} we obtain that for all $U\in \{Y_{t_{i-1}}, \tY^\pi_{t_{i-1}}\}$ and $V\in \{Y_{t_i}-Y_{t_{i-1}}, \tY^\pi_{t_i}-\tY^\pi_{t_{i-1}}\}$,
	\[
	\EE[ UV] = 0.
	\]
	As a consequence we get $\EE[ (Y_{t_{i-1}}- \tY^\pi_{t_{i-1}})\, \alpha_i] =0$, and therefore 
	\begin{equation}\label{q1}
		\Delta_i^2 = \Delta_{i-1}^2 + 2\EE[ (Y_{t_{i-1}}- \tY^\pi_{t_{i-1}})\, \alpha_i] + \EE[\alpha_i^2] =  \Delta_{i-1}^2 + \EE[\alpha_i^2].
	\end{equation}
	
	By the smoothness properties of the function $G_\mu$, see Lemma~\ref{lem1}, we derive
	\begin{align*}
		\alpha_i & = (G_\mu(X_{t_i}) - G_\mu(\tX^\pi_{t_i})) -  (G_\mu(X_{t_{i-1}}) - G_\mu(\tX^\pi_{t_{i-1}})) \\
		& = \int_{\tX^\pi_{t_i}}^{X_{t_i}} G_\mu'(u)\, du -  \int_{\tX^\pi_{t_{i-1}}}^{X_{t_{i-1}}} G_\mu'(u)\, du \\
		& = \beta_i + \gamma_i +\delta_i,
	\end{align*}
	where
	\begin{align*}
		\beta_i & = \int_{\tX^\pi_{t_i}}^{X_{t_i}} (G_\mu'(u) - G_\mu'(\tX^\pi_{t_i}))\, du -  \int_{\tX^\pi_{t_{i-1}}}^{X_{t_{i-1}}} (G_\mu'(u)-G_\mu'(\tX^\pi_{t_{i-1}}))\, du,\\
		\gamma_i & = \bigl(G_\mu'(\tX^\pi_{t_i})-G_\mu'(\tX^\pi_{t_{i-1}})\bigr)\, (X_{t_{i-1}} - \tX^\pi_{t_{i-1}}),\\
		\delta_i & = G_\mu'(\tX^\pi_{t_i})\, \bigl( (X_{t_i} - \tX^\pi_{t_i}) - (X_{t_{i-1}} - \tX^\pi_{t_{i-1}})\bigr).
	\end{align*}
	Hence
	\begin{equation}\label{q2}
		\EE\bigl[\alpha_i^2\bigr] \ge \frac{1}{2}\EE\bigl[\delta_i^2\bigr] - \EE\bigl[(\beta_i + \gamma_i)^2\bigr]
		\ge \frac{1}{2}\EE\bigl[\delta_i^2\bigr] - 2\EE\bigl[\beta_i^2\bigr] -2\EE\bigl [\gamma_i^2\bigr].
	\end{equation}
	
	We proceed by estimating $\EE[\beta_i^2]$. Employing again Lemma~\ref{lem1} we get
	\begin{equation}\label{q3}
		\begin{aligned}
			\beta_i^2& = \biggl| -2\int_{\tX^\pi_{t_i}}^{X_{t_i}} \int_{\tX^\pi_{t_i}}^u\mu(v)G_\mu'(v)\, dv\, du +2  \int_{\tX^\pi_{t_{i-1}}}^{X_{t_{i-1}}} \int_{\tX^\pi_{t_{i-1}}}^u \mu(v)G_\mu'(v)\, dv\, du\biggr|^2\\
			& \le \|\mu\|_\infty^2\, \|G_\mu'\|_\infty^2 \,\bigl( |X_{t_i}-\tX^\pi_{t_i}  |^2 +  |X_{t_{i-1}}-\tX^\pi_{t_{i-1}}|^2\bigr)^2.
		\end{aligned}
	\end{equation} 
	Let $X^\text{E}_{4n}$ and $\tX^{\pi,\text{E}}_{4n}$ denote the continuous-time Euler schemes with $4n$ equidistant steps for $X$ and $\tX^\pi$, respectively, and let $\uti=\lfloor t_i 4n\rfloor/(4n)$. 
	By~\eqref{qx1} we have  $X^\text{E}_{4n,t} = \tX^{\pi,\text{E}}_{4n,t}$ for every 
	$t\in \{j/(4n)\colon j\in\{0,\dots, 4n\}\}$
	and therefore
	\begin{align*}
		X_{t_i}-\tX^\pi_{t_i} & = (X_{\uti} - \tX^\pi_{\uti}) +	(X_{t_i} - X_{\uti } ) - (\tX^\pi_{t_i} - \tX^\pi_{\uti } ) \\
		& =(X_{\uti} -X^\text{E}_{4n,\uti})  + (\tX^{\pi,\text{E}}_{4n,\uti}-\tX^\pi_{\uti}) + \int_{\uti}^{t_i} ( \mu(X_s) - \mu(\tX^\pi_s))\, ds,
	\end{align*}
	which implies
	\begin{equation}\label{rst44}
		|X_{t_i}-\tX^\pi_{t_i}| \le \|X -X^\text{E}_{4n}\|_\infty  + \|\tX^\pi-\tX^{\pi,\text{E}}_{4n}\|_\infty + \frac{\|\mu\|_\infty}{2n}.
	\end{equation}

	Put
	\[
	\delta= \eps = (s-1/2-c_0)/6.
	\]
	Then $\delta,\eps \in (0,\infty)$ and $s-\delta-1/4 >0$. Since $\mu\in \Wc^{(s-1/4)-,4}$ we have $\mu\in \Wc^{s-\delta-1/4,4}$ and therefore we may apply                                                                           
	Theorem~\ref{thm1} with $p=4$ and $s-\delta-1/4$ in place of $s$  to obtain
	\begin{equation}\label{rst55}
		\EE\bigl[ \|X -X^\text{E}_{4n}\|^4_\infty\bigr] = \EE\bigl[ \|\tX^\pi-\tX^{\pi,\text{E}}_{4n}\|^4_\infty\bigr]  \le \frac{c}{n^{2(1+s-\delta-1/4)-4 \eps}} = \frac{c}{n^{2+s + c_0}}.
	\end{equation}
	Note that $2+s + c_0 < 2+2s -1/2 < 4$. Combining~\eqref{rst44} with~\eqref{rst55} thus yields 
	\[
	\EE\bigl[	|X_{t_i}-\tX^\pi_{t_i}|^4\bigr] \le \frac{c}{n^{2+s + c_0}}
	\]
	and the same estimate holds with $t_i$ replaced by $t_{i-1}$. Combining these estimates with the bound~\eqref{q3} we conclude that 
	\begin{equation}\label{q4}
		\EE\bigl[\beta_i^2\bigr] \le \frac{c}{n^{2+s +c_0}}.
	\end{equation}

	Next, we provide an upper bound for $\EE\bigl[\gamma_i^2\bigr]$. Clearly, $X_{t_{i-1}}$ and $\tX^\pi_{t_{i-1}}$ are measurable with respect to $\Fc^{W,\tW^\pi}_{t_{i-1}}$, and therefore,
	\begin{equation}\label{q5}
		\EE\bigl[\gamma_i^2\bigr] = \EE\bigl[(X_{t_{i-1}}-\tX^\pi_{t_{i-1}})^2\, \EE\bigl[ (G_\mu'(\tX^\pi_{t_i})-G_\mu'(\tX^\pi_{t_{i-1}}))^2\bigr|\Fc^{W,\tW^\pi}_{t_{i-1}}\bigr]\bigr].
	\end{equation}
	By Lemma~\ref{lem1} we obtain
	\begin{equation}\label{q6}
		(X_{t_{i-1}}-\tX^\pi_{t_{i-1}})^2 = (G_\mu^{-1}(Y_{t_{i-1}})-G_\mu^{-1}(\tY^\pi_{t_{i-1}}))^2 \le \|(G_\mu^{-1})'\|_\infty^2\, 	(Y_{t_{i-1}}-\tY^\pi_{t_{i-1}})^2
	\end{equation}
	as well as
	\begin{equation}\label{q7}
		\begin{aligned}
			(G_\mu'(\tX^\pi_{t_i})-G_\mu'(\tX^\pi_{t_{i-1}}))^2 & = \biggl(\int_{\tX^\pi_{t_{i-1}}}^{\tX^\pi_{t_i}} G_\mu''(u)\, du\biggr)^2 \\
			& \le 4\|\mu\|_\infty^2 \|G_\mu'\|_\infty^2\, (\tX^\pi_{t_{i}}-\tX^\pi_{t_{i-1}})^2\\
			& = 4\|\mu\|_\infty^2 \|G_\mu'\|_\infty^2 \biggl(\int_{t_{i-1}}^{t_i} \mu (\tX^\pi_t)\, dt + W_{t_{i}}-W_{t_{i-1}}\biggr)^2\\
			& \le 8\|\mu\|_\infty^2 \|G_\mu'\|_\infty^2 \left( \frac{\|\mu\|_\infty^2}{ 16n^2} +  (W_{t_{i}}-W_{t_{i-1}})^2\right).
		\end{aligned}
	\end{equation}
	Using~\eqref{qx2} we derive from~\eqref{q7} that
	\begin{equation}\label{q8}
		\begin{aligned}
			\EE\bigl[ (G_\mu'(\tX^\pi_{t_i})-G_\mu'(\tX^\pi_{t_{i-1}}))^2\bigr|\Fc^{W,\tW^\pi}_{t_{i-1}}\bigr] 
			& \le c \left(\frac{1}{ n^2}+ \EE[ (W_{t_{i}}-W_{t_{i-1}})^2|\Fc^{W,\tW^\pi}_{t_{i-1}}]\right)\\
			&\le c \left( \frac{1}{ n^2} + \frac{1}{4n}\right).
		\end{aligned}	
	\end{equation}
	Combining~\eqref{q5}, ~\eqref{q6} and~\eqref{q8} we get 
	\begin{equation}\label{q9}
		\EE\bigl[\gamma_i^2\bigr] \le \frac{c}{n} \Delta_{i-1}^2.	
	\end{equation}
	Combining~\eqref{q2}, ~\eqref{q4} and~\eqref{q9} yields 
	\begin{equation}\label{q10}
		\EE\bigl[\alpha_i^2\bigr] \ge \frac{1}{2}\EE[\delta_i^2] - \frac{c}{n} \Delta_{i-1}^2 - \frac{c}{n^{2+s+c_0}}.
	\end{equation}
	Finally, combine~\eqref{q1} with \eqref{q10}, use 
	\[
	\EE\bigl[\delta_i^2\bigr] \ge \inf_{x\in\R} |G_\mu'(x)|^2\, \EE\bigl[|(X_{t_i}- X_{t_{i-1}}) - (\tX^\pi_{t_i}- \tX^\pi_{t_{i-1}})|^2\bigr] 
	\] 
	and observe that $\inf_{x\in\R} |G_\mu'(x)| >0$, see Lemma~\ref{lem1}, to complete the proof of the lemma.
\end{proof}

We proceed with the analysis of the term $\EE\bigl[|(X_{t_i}- X_{t_{i-1}}) - (\tX^\pi_{t_i}- \tX^\pi_{t_{i-1}})|^2\bigr] $ 
in the case of $\mu = \mu_s$.

\begin{lemma}\label{lemf4}
	Let $s\in (1/2,1)$ and let $\mu = \mu_s$. Then there exist $c_0,c_1\in (0,\infty)$ such that for  all $n\in\N$, all $\pi=\{t_1,\dots,t_{5n}\}\in\tPi^n$ with
	$0<t_1<\dots <t_{5n}=1$ and all $ i\in \{1,\dots,5n\}$ with $t_{i-1} \ge 1/2$,
	\begin{equation}\label{iter2a}
		\begin{aligned}
			& \EE\bigl[|(X_{t_i}- X_{t_{i-1}}) - (\tX^\pi_{t_i}- \tX^\pi_{t_{i-1}})|^2\bigr] \\
			& \qquad \ge \frac{1}{2} \EE\Bigl[ \Bigl| \int_{t_{i-1}}^{t_i} \bigl (\mu_s (X_{t_{i-1} } +W_t - W_{t_{i-1} } ) -  \mu_s (X_{t_{i-1}} +\tW^\pi_t - \tW^\pi_{t_{i-1}})  \bigr)\,
			dt\Bigr|^2\Bigr] \\
			& \qquad\qquad - \frac{c_1}{ n}  \EE\bigl[|Y_{t_{i-1}} - \widetilde Y^\pi_{t_{i-1}}|^2\bigr]-\frac{c_1}{ n^{2+s+c_0}}.
		\end{aligned}
	\end{equation}
\end{lemma}

\begin{proof}
	Let $n\in\N$ and  $\pi=\{t_1,\dots,t_{5n}\}\in\tPi^n$ with $0<t_1<\dots <t_{5n}=1$ and fix $i\in\{1,\dots,n\}$. Throughout this proof we use $c\in (0,\infty)$ to denote a positive constant that 
	neither depends on  $n$ nor on $\pi$ nor on $i$. 
	The value of $c$ may change from 
	line to line.

	Put
	\begin{align*}
		A_i & = \int_{t_{i-1}}^{t_i} \bigl (\mu_s (X_{t_{i-1} } +W_t - W_{t_{i-1} } ) -  \mu_s (X_{t_{i-1}} +\widetilde W^\pi_t - \widetilde W^\pi_{t_{i-1}})  \bigr)\, dt,\\
		B_i  & = \int_{t_{i-1}}^{t_i} \bigl (\mu_s (X_t ) -  \mu_s (X_{t_{i-1}} +W_t - W_{t_{i-1}})  \bigr)\, dt,\\
		C_i & = \int_{t_{i-1}}^{t_i} \bigl (\mu_s (X_t ) -  \mu_s (\tX^\pi_t)  \bigr)\, dt,\\
		D_i  & = \int_{t_{i-1}}^{t_i} \bigl (\mu_s (\tX^\pi_{t_{i-1}} + \tW^\pi_t - \tW^\pi_{t_{i-1}}) -\mu_s(\tX^\pi_t)  \bigr)\, dt,\\
		E_i& = \int_{t_{i-1}}^{t_i} \bigl (\mu_s (X_{t_{i-1} } + \tW^\pi_t - \tW^\pi_{t_{i-1} } ) -  \mu_s (\tX^\pi_{t_{i-1}} +\widetilde W^\pi_t - \widetilde W^\pi_{t_{i-1}})  \bigr)\, dt.
	\end{align*}
	Using~\eqref{qx1} we get
	\[
	(X_{t_i}- X_{t_{i-1}}) - (\tX^\pi_{t_i}- \tX^\pi_{t_{i-1}}) = C_i = A_i +B_i +D_i + E_i
	\]
	and therefore,
	\begin{equation}\label{q11}
		\bigl((X_{t_i}- X_{t_{i-1}}) - (\tX^\pi_{t_i}- \tX^\pi_{t_{i-1}}) \bigr)^2 \ge \frac{1}{2}A_i^2 - 3(B_i^2+ D_i^2+ E_i^2).
	\end{equation}
	Clearly, $\PP^{B_i} = \PP^{-D_i}$. We thus conclude from~\eqref{q11}  that
	\begin{equation}\label{q12}
		\begin{aligned}
			\EE\bigl[|(X_{t_i}- X_{t_{i-1}}) - (\tX_{t_i}- \tX_{t_{i-1}})|^2\bigr] & \ge   \frac{1}{2}\EE\bigl[A_i^2\bigr] - 3\bigl(\EE\bigl[B_i^2\bigr] + \EE[D_i^2]+\EE\bigl[E_i^2\bigr]\bigr) \\
			& =   \frac{1}{2} \EE\bigl[A_i^2\bigr]- 6 \EE\bigl[B_i^2\bigr] -3\EE\bigl[E_i^2\bigr].
		\end{aligned}
	\end{equation}
	
	We proceed with estimating the term $\EE\bigl[B_i^2\bigr]$. Let $\xi\in (0,\infty)$, let 
	\[
	s' = (s+ \max(1/2,3s/4))/2
	\]
	and note that $s'\in (1/2,s)$.	
	For every $t\in [t_{i-1},t_i]$ put                                                                                    
	\[
	R_t = X_{t_{i-1}} + W_t-W_{t_{i-1}} 
	\]
	and note that $X_t-R_t =  \int_{t_{i-1}}^t \mu_s(X_u)\, du$. Using 
	Lemma~\ref{lem3}(v),(vi)
	we obtain that there exists $\tilde c\in (0,\infty)$ such that for all $x,y\in\R$,
	\[
	|\mu_s(x)-\mu_s(y)| \le \begin{cases}
		\tilde c\, \frac{|x-y|}{\xi}, & \text{if } x,y \ge \xi \text{ or }x,y\le -\xi,\\[.1cm]
		\tilde c\,|x-y|^{s'-1/2}, & \text{if }  -\xi \le x,y\le \xi,\\[.1cm]
		2\|\mu_s\|_\infty, & \text{if } (x-\xi)(y-\xi) < 0 \text{ or } (x+\xi)(y+\xi) < 0.
	\end{cases}
	\]
	Using the latter fact and boundedness of $\mu_s$ we conclude that
	\begin{equation}\label{q13}
		\begin{aligned}
			\EE\bigl[B_i^2\bigr] & \le c\, \EE\biggl[\biggl|\int_{t_{i-1}}^{t_i} \Bigl( \frac{1}{\xi}\,|X_t-R_t| + |X_t-R_t|^{s'-1/2} 1_{[-\xi,\xi]^2} \bigl(X_t,R_t\bigr) \\
			& \qquad\qquad \qquad + 2\|\mu_s\|_\infty \bigl( 1_{\{(X_t-\xi)(R_t-\xi) <0\}} + 1_{\{(X_t+\xi)(R_t+\xi) <0\}} \bigr)\Bigr)\, dt \biggr|^2\biggr]\\
			&  \le c\, \biggl(\EE \biggl[\biggl(\int_{t_{i-1}}^{t_i}\frac{1}{\xi}\,\biggl| \int_{t_{i-1}}^t \mu_s(X_u)\, du\biggr|\, dt\biggr)^2 \biggr]  \\
			& \qquad \qquad + \EE \biggl[\biggl(\int_{t_{i-1}}^{t_i}\biggl| \int_{t_{i-1}}^t \mu_s(X_u)\, du\biggr|^{s'-1/2}\, 1_{[-\xi,\xi]}(X_t) \,dt\biggr)^2 \biggr] \\
			& \qquad \qquad\qquad +   \EE \biggl[\biggl(\int_{t_{i-1}}^{t_i}\bigl( 1_{\{(X_t-\xi)(R_t-\xi) <0\}} + 1_{\{(X_t+\xi)(R_t+\xi) <0\}} \bigr)\, dt \biggr)^2\biggr] \biggr)\\
			& \le c\biggl( \frac{1}{\xi^2 n^4} + \frac{1}{n^{2s'-1}}\, \EE\biggl[\biggl(\int_{t_{i-1}}^{t_i} 1_{[-\xi,\xi]}(X_t) \,dt\biggr)^2 \biggr]\\
			& \qquad \qquad  +   \EE \biggl[\biggl(\int_{t_{i-1}}^{t_i}\bigl( 1_{\{(X_t-\xi)(R_t-\xi) <0\}} + 1_{\{(X_t+\xi)(R_t+\xi) <0\}} \bigr)\, dt \biggr)^2\biggr] \biggr).
		\end{aligned}
	\end{equation}	
	By Lemma~\ref{Xlem1}(ii),
	\begin{equation}\label{q14}
		\EE\biggl[\biggl(\int_{t_{i-1}}^{t_i} 1_{[-\xi,\xi]}(X_t) \,dt\biggr)^2 \biggr] \le (t_i-t_{i-1}) \int_{t_{i-1}}^{t_i} \PP(|X_t| \le \xi)\, dt \le c\,\xi\, (t_i-t_{i-1})^2 \le c\frac{\xi}{n^2}.
	\end{equation}
	Moreover, by  Lemma~\ref{Xlem1}(iv),
	\begin{equation}\label{q15}
		\EE \biggl[\biggl(\int_{t_{i-1}}^{t_i}\bigl( 1_{\{(X_t-\xi)(R_t-\xi) <0\}} + 1_{\{(X_t+\xi)(R_t+\xi) <0\}} \bigr)\, dt \biggr)^2\biggr] \le c\,  (t_i-t_{i-1})^3 \le \frac{c}{n^3}.	
	\end{equation}
	Combining ~\eqref{q13} to~\eqref{q15} yields 
	\begin{equation}\label{q16a}
		\EE \bigl[B_i^2\bigr]  \le c \Bigl( \frac{1}{\xi^2 n^4} +  \frac{\xi}{ n^{1+2s'}} + \frac{1}{ n^3}\Bigr)
	\end{equation}
	and
	taking the minimizer 
	$\xi = 2^{1/3}n^{-1+2s'/3}$ 
	in~\eqref{q16a} we obtain
	\begin{equation}\label{q16aa}
		\EE \bigl[B_i^2\bigr]  \le c\, \Bigl( \frac{1}{n^{2+4s'/3}} + \frac{1}{ n^3}\Bigr).
	\end{equation}
	Note that 
	\[
	4s'/3 = \max\bigl((2s+1)/3,7s/6\bigr) = s +\max((1-s)/3, s/6)
	\]
	and therefore  we finally get from~\eqref{q16aa} that
	\begin{equation}\label{q16}
		\EE\bigl [B_i^2\bigr]  \le \frac{c}{ n^{2+s+c_0}},
	\end{equation}
	where $c_0= \min\bigl(\max((1-s)/3,s/6),1-s\bigr) \in (0,\infty)$.
	
	Next, we estimate the term $\EE\bigl[E_i^2\bigr]$. By~\eqref{qx2} we have for $\PP^{(X_{t_{i-1}}, \tX^\pi_{t_{i-1}})}$-almost all $(x,\tx) \in \R^2$,
	\begin{equation}\label{q17}
		\begin{aligned}
			&	\EE[E_i^2\,| \, (X_{t_{i-1}}, \tX^\pi_{t_{i-1}}) = ( x,\tx)]\\  & \qquad\qquad = \EE\Big[ \Bigl| \int_{t_{i-1}}^{t_i} \bigl(\mu_s(x+W_t-W_{t_{i-1}}) -\mu_s(\tx+W_t-W_{t_{i-1}}) \bigr)\, dt\Bigr|^2  \Bigr]\\
			&\qquad\qquad  =  \EE\Big[ \Bigl| \int_0^{t_{i}-t_{i-1}} \bigl(\mu_s(x+W_t) -\mu_s(\tx+W_t) \bigr)\, dt\Bigr|^2  \Bigr].
		\end{aligned}
	\end{equation}
	Note that for all $w\in\R$,
	\[
	\mu_s(x+w) -\mu_s(\tx+w) = \int_{\R} h_s(z)\, \bigl(e^{\bi z(x+w)} - e^{\bi z(\tx+w)}\bigr)\, dz.
	\]
	Hence, by~\eqref{q17},  for $\PP^{(X_{t_{i-1}}, \tX^\pi_{t_{i-1}})}$-almost all $(x,\tx) \in \R^2$,
	\begin{equation}\label{q18}
		\begin{aligned}
			&	\EE\bigl[E_i^2\,| \, (X_{t_{i-1}}, \tX^\pi_{t_{i-1}}) = ( x,\tx)\bigr]\\  & \quad = 2  
			\int_0^{t_i-t_{i-1}} \int_{t}^{t_i-t_{i-1}} \EE\Big[\bigl(\mu_s(x+W_t) -\mu_s(\tx+W_t) \bigr)\, \bigl(\mu_s(x+W_r) -\mu_s(\tx+W_r) \bigr)\Bigr]\, dr\, dt\\
			&\quad  = 2  \int_0^{t_i-t_{i-1}} \int_{t}^{t_i-t_{i-1}}\EE\Bigl[  \int_\R\int_\R h_s(u) h_s(v) \bigl(e^{\bi u(x+W_t)} - e^{\bi u(\tx+W_t)}\bigr) \\
			& \qquad\qquad\qquad\qquad\qquad\qquad\qquad\qquad\qquad\qquad \bigl(e^{\bi v(x+W_r)} - e^{\bi v(\tx+W_r)}\bigr)\, du\, dv \Bigr] \, dr \, dt\\
			&\quad  = 2 \int_0^{t_i-t_{i-1}} \int_{t}^{t_i-t_{i-1}} \int_\R\int_\R h_s(u) h_s(v) \bigl(e^{\bi ux} - e^{\bi u\tx}\bigr) \bigl(e^{\bi vx} - e^{\bi v\tx}\bigr) \\
			& \qquad\qquad\qquad\qquad\qquad\qquad\qquad\qquad\qquad\qquad \EE\bigl [ e^{\bi (uW_t + vW_r)}\bigr] \, du\, dv\, dr \, dt\\
			&\quad  =2 \int_0^{t_i-t_{i-1}} \int_{t}^{t_i-t_{i-1}} \int_\R\int_\R h_s(u) h_s(v) \bigl(e^{\bi ux} - e^{\bi u\tx}\bigr) \bigl(e^{\bi vx} - e^{\bi v\tx}\bigr)\\
			& \qquad\qquad\qquad\qquad\qquad\qquad\qquad\qquad\qquad\qquad e^{-t(u+v)^2/2} \, e^{-(r-t) v^2/2} \, du\, dv\, dr \, dt.
		\end{aligned}
	\end{equation}
	
	Note that  $|e^{\bi ux} - e^{\bi u\tx}| \le |u||x-\tx|$ and $|u||h_s(u)| \le |u|/(e+|u|) \le 1$ for all $u\in\R$. By the latter two
	estimates and~\eqref{qx2} we derive from~\eqref{q18} that for $\PP^{(X_{t_{i-1}}, \tX^\pi_{t_{i-1}})}$-almost all $(x,\tx) \in \R^2$,
	\begin{equation}\label{q19}
		\begin{aligned}
			&	\EE\bigl[E_i^2\,| \, (X_{t_{i-1}}, \tX^\pi_{t_{i-1}})= ( x,\tx)\bigr] \\
			&\qquad\qquad\ \le  2 |x-\tx|^2\int_0^{t_i-t_{i-1}} \int_{t}^{t_i-t_{i-1}} \int_\R\int_\R e^{-t(u+v)^2/2} \, e^{-(r-t) v^2/2} \, du\, dv\, dr \, dt\\
			& \qquad\qquad = 4\pi |x-\tx|^2\int_0^{t_i-t_{i-1}} \int_{t}^{t_i-t_{i-1}}  \frac{1}{\sqrt{t}} \frac{1}{\sqrt{r-t}}\, dr\, dt\\
			& \qquad\qquad\ \le 4\pi |x-\tx|^2 \bigl(2\sqrt{t_i-t_{i-1}}\bigr)^2 \le 4\pi |x-\tx|^2\cdot \frac{1}{n}.
		\end{aligned}
	\end{equation}
	Taking expectations in~\eqref{q19} and using Lemma~\ref{lem1}(ii) we obtain 
	\begin{equation}\label{q21}
		\EE\bigl[E_i^2\bigr]\le \frac{c}{n}\,\EE\bigl[|X_{t_{i-1}} - \tX^\pi_{t_{i-1}}|^2\bigr] \le \frac{c}{n}\, \EE\bigl[|Y_{t_{i-1}} - \widetilde Y^\pi_{t_{i-1}}|^2\bigr].
	\end{equation}
	Finally, combine~\eqref{q12} with~\eqref{q16} and~\eqref{q21} to obtain~\eqref{iter2a}.
\end{proof}

Combining Lemma~\ref{lemf3} with Lemma~\ref{lemf4}  we obtain the following 
lower bound for the quantity  $\EE\bigl[|Y_{1} - \widetilde Y^\pi_{1}|^2\bigr]$.

\begin{lemma}\label{lemf3b}
	Let $s\in (1/2,1)$ and let $\mu = \mu_s$. Then there exist $c_0,c_1\in (0,\infty)$ and $n^*\in\N$ such that for  all $n\in\N$ with $n\ge n^*$ and all $\pi=\{t_1,\dots,t_{5n}\}\in\tPi^n$ with $0<t_1<\dots <t_{5n}=1$,
	\begin{equation}\label{iterx}
		\begin{aligned}
			&  \EE\bigl[|Y_{1} - \widetilde Y^\pi_{1}|^2\bigr] \\
			&\qquad  \ge c_1 \sum_{i=i^*(\pi)}^{5n}  \EE\Bigl[ \Bigl|\int_{t_{i-1}}^{t_i} \bigl (\mu_s (X_{t_{i-1} } +W_t - W_{t_{i-1} } ) -  \mu_s (X_{t_{i-1}} +
			\widetilde W^\pi_t - \widetilde W^\pi_{t_{i-1}})  \bigr)\, dt\Bigr|^2\Bigr]\\
			& \qquad\qquad - \frac{1}{n^{1+s+c_0}},
		\end{aligned} 
	\end{equation}
	where $i^*(\pi) = \min\{i\in \{1,\dots,5n\}\colon t_{i-1} \ge 1/2\}$.
\end{lemma}

\begin{proof}
	Let $s\in (1/2,1)$,  let $n\in\N$  and let $\pi=\{t_1,\dots,t_{5n}\}\in\tPi^n$ with $0<t_1<\dots <t_{5n}=1$. 
	
	By Lemma~\ref{lem3} we have $\mu_s\in \Lc^1(\R)\cap\Wc^{(s-1/4)-,4}$. We may thus apply Lemma~\ref{lemf3} with $c_0 = (s-1/2)/2$  to obtain that there exist $c'_1,c'_2\in (0,\infty)$, 
	which neither depend on  $n$ nor on $\pi$,
	such that for all  $i\in \{1,\dots,5n\}$,
	\begin{equation}\label{q22}
		\begin{aligned}
			\EE\bigl[|Y_{t_i} - \widetilde Y^\pi_{t_i}|^2\bigr] & \ge \Bigl(1-\frac{c'_1}{n}\Bigr)\, \EE\bigl[|Y_{t_{i-1}} - \widetilde Y^\pi_{t_{i-1}}|^2\bigr]  \\
			& \qquad\qquad + c'_2\, \EE\bigl[|(X_{t_i}- X_{t_{i-1}}) - (\tX^\pi_{t_i}- \tX^\pi_{t_{i-1}})|^2\bigr] - \frac{c'_1}{n^{2+s+(s-1/2)/2}}.
		\end{aligned}
	\end{equation}	
	Using Lemma~\ref{lemf4} we obtain the existence of $\tilde c_0,\tilde c_1\in (0,\infty)$, 
	which neither depend on  $n$ nor on $\pi$,
	such that for all $ i\in \{i^*(\pi),\dots,5n\}$,
	\begin{equation}\label{q23}
		\begin{aligned}
			& 	 \EE\bigl[|(X_{t_i}- X_{t_{i-1}}) - (\tX^\pi_{t_i}- \tX^\pi_{t_{i-1}})|^2\bigr] \\
			&\qquad\qquad  \ge \frac{1}{2}\,\EE\Bigl[ \Bigl|\int_{t_{i-1}}^{t_i} \bigl (\mu_s (X_{t_{i-1} } +W_t - W_{t_{i-1} } ) -  \mu_s (X_{t_{i-1}} +\widetilde W^\pi_t - \widetilde W^\pi_{t_{i-1}})  \bigr)\, dt\Bigr|^2\Bigr]\\ 
			& \qquad \qquad\qquad\qquad -\frac{\tilde c_1}{n}\EE\bigl[|Y_{t_{i-1}} - \widetilde Y^\pi_{t_{i-1}}|^2\bigr]    -\frac{\tilde c_1}{ n^{2+s+\tilde c_0}}.
		\end{aligned}
	\end{equation}			
	Combining~\eqref{q22} and~\eqref{q23} yields that for all $ i\in \{i^*(\pi),\dots,5n\}$,
	\begin{equation}\label{q23a}
		\begin{aligned}
			\EE\bigl[|Y_{t_i} - \widetilde Y^\pi_{t_i}|^2\bigr] & \ge \bar c_1\, \EE\Bigl[ \Bigl|\int_{t_{i-1}}^{t_i} \bigl (\mu_s (X_{t_{i-1} } +W_t - W_{t_{i-1} } ) -  \mu_s (X_{t_{i-1}} +\widetilde W^\pi_t - \widetilde W^\pi_{t_{i-1}})  \bigr)\, dt\Bigr|^2\Bigr] \\
			& \qquad\qquad 	+	\Bigl(1-\frac{\bar c_2}{n}\Bigr)\, \EE\bigl[|Y_{t_{i-1}} - \widetilde Y^\pi_{t_{i-1}}|^2\bigr]  - \frac{\bar c_2}{n^{2+s+ \bar c_0}},
		\end{aligned}
	\end{equation}		
	where  $\bar c_0=\min(\tilde c_0, (s-1/2)/2)$, $\bar c_1=c'_2/2$ and $\bar c_2= \tilde c_1 c_2' + c'_1$.
	
	Take $n^*\in \N$ such that for all $n\in \N$ with $n\ge n^*$,
	\begin{equation}\label{bedn}
		n > \bar c_2\quad \text{ and }\quad \Bigl(1-\frac{\bar c_2}{n}\Bigr)^n \ge e^{-\bar c_2}/2. 
	\end{equation}
	Assume $n\ge n^*$. Then, for all $i\in\{0,\dots,5n\}$,
	\[
	\Bigl(1-\frac{\bar c_2}{n}\Bigr)^{i} \ge \Bigl(1-\frac{\bar c_2}{n}\Bigr)^{5n} \ge (e^{-\bar c_2}/2)^5,
	\]
	and, furthermore,
	\[
	\sum_{i=0}^{5n} \Bigl(1-\frac{\bar c_2}{n}\Bigr)^{i} \le \sum_{i=0}^{\infty}  \Bigl(1-\frac{\bar c_2}{n}\Bigr)^{i} = \frac{n}{\bar c_2}.
	\]
	Iteratively applying~\eqref{q23a} for $i=5n,5n-1,\dots,i^*(\pi)$ we thus obtain
	\begin{equation}\label{q23b}
		\begin{aligned}
			\EE\bigl[|Y_{1} - \widetilde Y^\pi_{1}|^2\bigr] &  \ge \bar c_1 \sum_{i=i^*(\pi)}^{5n}  \Bigl(1-\frac{\bar c_2}{n}\Bigr)^{5n-i}\, \EE\Bigl[ \Bigl|\int_{t_{i-1}}^{t_i} \bigl (\mu_s (X_{t_{i-1} } +W_t - W_{t_{i-1} } ) \\
			&\qquad\qquad\qquad\qquad \qquad\qquad\qquad \qquad\qquad-  \mu_s (X_{t_{i-1}} +\widetilde W^\pi_t - \widetilde W^\pi_{t_{i-1}})  \bigr)\, dt\Bigr|^2\Bigr] \\
			& \qquad\qquad +	\Bigl(1-\frac{\bar c_2}{n}\Bigr)^{5n-i^*(\pi)+1}\, \EE\bigl[|Y_{t_{i^*(\pi)-1}} - \widetilde Y^\pi_{t_{i^*(\pi)-1}}|^2\bigr] \\
			&\qquad\qquad\qquad	- \frac{\bar c_2}{n^{2+s+\bar c_0}}\, \sum_{i=i^*(\pi)}^{5n}  \Bigl(1-\frac{\bar c_2}{n}\Bigr)^{5n-i}\\
			&  \ge \bar c_1 (e^{-\bar c_2}/2)^5  \sum_{i=i^*(\pi)}^{5n} \EE\Bigl[ \Bigl|\int_{t_{i-1}}^{t_i} \bigl (\mu_s (X_{t_{i-1} } +W_t - W_{t_{i-1} } ) \\
			&\qquad\qquad\qquad\qquad \qquad\qquad-  \mu_s (X_{t_{i-1}} +\widetilde W^\pi_t - \widetilde W^\pi_{t_{i-1}})  \bigr)\, dt\Bigr|^2\Bigr] 	- \frac{1}{n^{1+s+\bar c_0}}.
		\end{aligned}
	\end{equation}		
\end{proof}

Next, we provide a lower bound for the term  $\EE\bigl[|\int_{t_{i-1}}^{t_i} (\mu_s (X_{t_{i-1} } +W_t - W_{t_{i-1} } ) -  \mu_s (X_{t_{i-1}} +\widetilde W^\pi_t - \widetilde W^\pi_{t_{i-1}})  )\, dt|^2\bigr]$ in~\eqref{iterx}.

\begin{lemma}\label{lemf5}
	Let $s\in (1/2,1)$ and let $\mu = \mu_s$. Then there exist $n^\ast\in\N$ and $c\in (0,\infty)$ such that for  all $n\in\N$ with $n\ge n^\ast$, all $\pi=\{t_1,\dots,t_{5n}\}\in\tPi^n$ with $0<t_1<\dots <t_{5n}=1$ and all $ i\in \{1,\dots,5n\}$ with $t_{i-1} \ge 1/2$ and $t_i-t_{i-1} = 1/(4n)$,
	\begin{equation}\label{iter2b}
		\EE\Bigl[\Bigl| \int_{t_{i-1}}^{t_i} \bigl (\mu_s (X_{t_{i-1} } +W_t - W_{t_{i-1} } ) -  \mu_s (X_{t_{i-1}} +\widetilde W^\pi_t - \widetilde W^\pi_{t_{i-1}})  \bigr)\, dt \Bigr|^2\Big] \ge \frac{ c}{\ln^2(n+1)n^{2+s}}.
	\end{equation}
\end{lemma}

\begin{proof}
	Let $n\in\N$ and  $\pi=\{t_1,\dots,t_{5n}\}\in\tPi^n$ with $0<t_1<\dots <t_{5n}=1$ and fix $i\in\{1,\dots,n\}$ with $t_{i-1} \ge 1/2$ and $t_i-t_{i-1} = 1/(4n)$.  Throughout this proof we use $c_1,c_2,\dots\in (0,\infty)$ to denote  positive constants,
	which neither depend on  $n$ nor on $\pi$ nor on $i$.	
	
	Since $X_{t_{i-1}}$ is measurable with respect to $\Fc^{W,\tW^\pi}_{t_{i-1}}$ we obtain by~\eqref{qx2} that
	\begin{equation}\label{rsv1}
		\begin{aligned}
			&	 \EE\Bigl[\Bigl| \int_{t_{i-1}}^{t_i} \bigl (\mu_s (X_{t_{i-1} } +W_t - W_{t_{i-1} } ) -  \mu_s (X_{t_{i-1}} +\widetilde W^\pi_t - \widetilde W^\pi_{t_{i-1}})  \bigr)\, dt \Bigr|^2\Big] \\
			& \qquad\qquad = \int_{\R}  \EE\Bigl[\Bigl| \int_{t_{i-1}}^{t_i} \bigl (\mu_s (x +W_t - W_{t_{i-1} } ) -  \mu_s (x +\widetilde W^\pi_t - \widetilde W^\pi_{t_{i-1}})  \bigr)\, dt \Bigr|^2\Big] \PP^{X_{t_{i-1}}} (dx).
		\end{aligned}
	\end{equation}		
	By Lemma~\ref{Xlem1}(i), the distribution of $X_{t_{i-1}}$ has a Lebesgue density $f_{X_{t_{i-1}}}$ that satisfies 
	\[
	f_{X_{t_{i-1}}}(x) \ge c_1\,\frac{1}{\sqrt{t_{i-1}}}\, e^{-\frac{c_2(x-x_0)^2}{t_{i-1}}}
	\]
	for all $x\in \R$. Since $t_{i-1}\in[1/2,1]$ and $e^{-2c_2(x-x_0)^2} \ge e^{-4c_2x_0^2} e^{-4c_2x^2}$ for all $x\in \R$ we thus obtain 
	\begin{equation}\label{rsv4}
		\begin{aligned}
			& \int_{\R}  \EE\Bigl[\Bigl| \int_{t_{i-1}}^{t_i} \bigl (\mu_s (x +W_t - W_{t_{i-1} } ) -  \mu_s (x +\widetilde W^\pi_t - \widetilde W^\pi_{t_{i-1}})  \bigr)\, dt \Bigr|^2\Big] \PP^{X_{t_{i-1}}} (dx)\\
			& \qquad \ge c_3\,  \int_{\R} e^{-c_4\,x^2}\EE\Bigl[\Bigl| \int_{t_{i-1}}^{t_i} \bigl (\mu_s (x +W_t - W_{t_{i-1} } ) -  \mu_s (x +\widetilde W^\pi_t - \widetilde W^\pi_{t_{i-1}})  \bigr)\, dt  \Bigr|^2\Big] \, dx.
		\end{aligned}
	\end{equation}	
	Note that for all $t\in [t_{i-1}, t_i]$ we have $W_t-W_{t_{i-1}} = \frac{t-t_{i-1}}{1/(4n)} (W_{t_i}- W_{t_{i-1}}) + B^\pi_t$ and, observing~\eqref{qx1},   $\tW^\pi_t-\tW^\pi_{t_{i-1}} =  \frac{t-t_{i-1}}{1/(4n)} (W_{t_i}- W_{t_{i-1}}) + \tB^\pi_t$. Hence, for all $x\in\R$,
	\begin{equation}\label{rsv2}
		\begin{aligned}
			& \int_{t_{i-1}}^{t_i} \bigl (\mu_s (x +W_t - W_{t_{i-1} } ) -  \mu_s (x +\widetilde W^\pi_t - \widetilde W^\pi_{t_{i-1}})  \bigr)\, dt \\
			& \qquad\qquad =\frac{1}{4n}  \int_0^1 \Bigl( \mu_s\bigl(x+ t(W_{t_i} - W_{t_{i-1}})   + B^\pi_{t_{i-1} + t/(4n)}  ) \\
			& \qquad\qquad\qquad\qquad\qquad\qquad - \mu_s\bigl(x+ t(W_{t_i} - W_{t_{i-1}})   + \tB^\pi_{t_{i-1} + t/(4n)} \bigr) \Bigr)\, dt.
		\end{aligned}
	\end{equation}	
	
	Let $\pistar = \{1\}$.
	By the independence of $(W_{t_i} - W_{t_{i-1}})$, 
	$(B^\pi_{t_{i-1} + t/(4n)})_{t\in [0,1]}$,  $(\tB^\pi_{t_{i-1} + t/(4n)})_{t\in [0,1]}$ 
	and the independence of $W_1$, $(B_t^{\pi^\ast})_{t\in [0,1]}$, $(\tB_t^{\pi^\ast})_{t\in [0,1]}$ we have
	\begin{align*}
		\PP^{ (t(W_{t_i} - W_{t_{i-1}}) +B^\pi_{t_{i-1} + t/(4n)}, t(W_{t_i} - W_{t_{i-1}})+ \tB^\pi_{t_{i-1} + t/(4n)})_{t\in [0,1]}  } & = \PP^{(\sqrt{1/(4n)} \, (tW_1+ B_t^{\pi^\ast}, tW_1+ \tB_t^{\pi^\ast} ))_{t\in [0,1]}} \\ & = \PP^{(\sqrt{1/(4n)} \, (W_t, \tW_t^{\pi^\ast} ))_{t\in [0,1]}}.
	\end{align*}
	Using~\eqref{rsv2} we may thus conclude that for all $x\in \R$,
	\begin{equation}\label{rsv3}
		\begin{aligned}
			&  \EE\Bigl[\Bigl| \int_{t_{i-1}}^{t_i} \bigl (\mu_s (x +W_t - W_{t_{i-1} } ) -  \mu_s (x +\widetilde W^\pi_t - \widetilde W^\pi_{t_{i-1}})  \bigr)\, dt \Bigr|^2\Big] \\
			& \qquad\qquad = \frac{1}{16n^2} \EE\Bigl[\Bigl| \int_0^1 \bigl (\mu_s (x + \sqrt{1/(4n)} \,W_t ) -  \mu_s (x +\sqrt{1/(4n)}\, \widetilde W^{\pi^\ast}_t )\bigr)\, dt \Bigr|^2\Big].
		\end{aligned}
	\end{equation}
	
	For  $z\in \R$ put
	\[
	\kappa(z) = \int_0^1 \bigl( e^{\bi zW_t} - e^{\bi z \tW^{\pi^\ast}_t}\bigr)\, dt
	\]
	and for later purposes note that for all $y,z\in \R$ we have
	\begin{equation}\label{rsv5a}
		|\kappa(z)| \le 2
	\end{equation}
	and 
	\begin{equation}\label{rsv5}
		\begin{aligned}
			|\kappa(z)-\kappa(y)| & = \Bigl|\int_0^1 \bigl( e^{\bi z W_t} - e^{\bi y W_t}\bigr)\, dt -                       \int_0^1 \bigl( e^{\bi z \tW^{\pi^*}_t} - e^{\bi y \tW^{\pi^*}_t}\bigr)\, dt\Bigr| \\
			& \le |z-y|\int_0^1 (|W_t| + |\tW^{\pi^*}_t|)\, dt.
		\end{aligned}
	\end{equation}
	Recall that $\mu_s = \F h_s$ a.e.
		 and $h_s\in \Lc^1(\R)$, see Section~\ref{mumu}. Pathwise applying Fubini's theorem we thus obtain for every $x\in \R$,
	\begin{equation}\label{rsv6}
		\begin{aligned}
			& \int_0^1 \bigl (\mu_s (x + \sqrt{1/(4n)} \,W_t ) -  \mu_s (x +\sqrt{1/(4n)}\, \widetilde W^{\pi^\ast}_t )\bigr)\, dt \\
			& \qquad \qquad =  \int_0^1 \int_{\R}  \bigl( e^{\bi z(x+ \sqrt{1/(4n)}\,W_t)} - e^{\bi z(x+ \sqrt{1/(4n)}\,\widetilde W^{\pi^\ast}_t)}\bigr)\, h_s(z)\, dz\, dt\\
			& \qquad \qquad = \int_{\R} h_s(z) e^{\bi z x} \kappa (\sqrt{1/(4n)}z)\, dz.
		\end{aligned}
	\end{equation}
	Combining~\eqref{rsv1}, \eqref{rsv4}, \eqref{rsv3} and \eqref{rsv6} we conclude that
	\begin{equation}\label{rsv7}
		\begin{aligned}
			&	 \EE\Bigl[\Bigl| \int_{t_{i-1}}^{t_i} \bigl (\mu_s (X_{t_{i-1} } +W_t - W_{t_{i-1} } ) -  \mu_s (X_{t_{i-1}} +\widetilde W^\pi_t - \widetilde W^\pi_{t_{i-1}})  \bigr)\, dt \Bigr|^2\Big] \\
			& \qquad\qquad \ge \frac{c_3}{16n^2}
			\EE\Bigl[ \int_{\R} e^{-c_4\,x^2} \Bigl| \int_\R h_s(z) \,e^{\bi z x} \kappa (\sqrt{1/(4n)}z)\, dz \Bigr|^2 \, dx\Bigr].
		\end{aligned}
	\end{equation}		
	
	Let $\omega\in\Omega$
	and define $f,g\colon \R\to \C$ by 
	\[
	f(x) =  e^{-c_4\,x^2/2},\quad g(x) = h_s(x)\,\kappa (\sqrt{1/(4n)} x,\omega)
	\]
	for  $x\in\R$. Clearly, $f\in \Lc^2(\R)$. Moreover, for Lebesgue almost all $x\in\R$,
	\[
	\F f(x) = \sqrt{\frac{2\pi }{c_4}} \,  e^{-x^2/(2c_4)},
	\] 
	which, in particular, yields $\F f\in \Lc^1(\R)\cap  \Lc^2(\R)$.
	Furthermore, by~\eqref{rsv5a} and $h_s\in \Lc^2(\R)$, we obtain $g\in   \Lc^2(\R)$ and therefore $\F g\in \Lc^2(\R)$. Hence, by Plancherel's theorem, the convolution theorem  and the Fourier inversion theorem, 
	\begin{equation}\label{rsv7a}
		\begin{aligned}		
			& \int_{\R} e^{-c_4\,x^2} \Bigl| \int_\R h_s(z) \,e^{\bi z x} \kappa (\sqrt{1/(4n)}z,\omega)\, dz \Bigr|^2 \, dx \\
			& \qquad\quad  = \int_\R |f(x)\,\F g(x)|^2\, dx\\
			& \qquad\quad  = 
			\frac{1}{2\pi}\,
			\int_\R |\F (f\cdot \F g )(x)|^2\, dx\\
			& \qquad\quad = \frac{1}{(2\pi)^3}\, \int_\R |\F f\ast \F\F g )(x)|^2\, dx\\
			& \qquad\quad = \frac{1}{(2\pi)^3}\,\int_\R  \Bigl| \int_\R \F f(u) \, \F\F g(x-u)\, du\Bigr|^2\, dx\\
			& \qquad\quad = \frac{1}{2\pi}\,\int_\R  \Bigl| \int_\R \F f(u) \, g(u-x)\, du\Bigr|^2\, dx\\
			& \qquad\quad = \frac{\sqrt{4n}}{2\pi}\, \int_\R  \Bigl| \int_\R \F f(u) \, g(u-\sqrt{4n}x)\, du\Bigr|^2\, dx\\
			& \qquad\quad  = \frac{1}{c_4(4n)^{s}}\,
			\int_\R  \Bigl| \int_\R   \frac{e^{-u^2/(2c_4)} \kappa(u/\sqrt{4n}-x,\omega)}{(e/\sqrt{4n} + |u/\sqrt{4n}-x|)^{1/2+s}\, \ln(e + |u-\sqrt{4n}x|)} 
			\, du\Bigr|^2\, dx.
		\end{aligned}
	\end{equation}
	
	Combining~\eqref{rsv7} with~\eqref{rsv7a} we conclude that
	\begin{equation}\label{rsv8a}
		\begin{aligned}		
			&\ln^2(n+1) \,  
			\EE\Bigl[\Bigl| \int_{t_{i-1}}^{t_i} \bigl (\mu_s (X_{t_{i-1} } +W_t - W_{t_{i-1} } ) -  \mu_s (X_{t_{i-1}} +\widetilde W^\pi_t - \widetilde W^\pi_{t_{i-1}})  \bigr)\, dt \Bigr|^2\Big]\\
			&\qquad\qquad \ge \frac{c_5}{n^{2+s}} \, \EE\Bigl[ \int_\R  \Bigl| \int_\R  R_n(u,x)\, du\Bigr|^2\, dx\Bigr],
		\end{aligned}
	\end{equation}
	where
	\begin{equation}\label{rsv8b}
		\begin{aligned}	
			R_n(u,x,\omega) & =  \frac{e^{-u^2/(2c_4)}\kappa(u/\sqrt{4n}-x,\omega)\ln(\sqrt{4n}) }{(e/\sqrt{4n} + |u/\sqrt{4n}-x|)^{1/2+s}\, \ln(e + |u-\sqrt{4n}x|)} 
		\end{aligned}
	\end{equation}
	for all $u,x\in\R$ and $\omega\in\Omega$.
	
	It remains to prove that
	\begin{equation}\label{rsv8}
		\begin{aligned}
			&	\liminf_{n\to\infty}  \EE\Bigl[ \int_\R  \Bigl| \int_\R  R_n(u,x)\, du\Bigr|^2\, dx\Bigr] >0.
		\end{aligned}	
	\end{equation}
	To this end we show below  that for all $x\in\R$, 
	$\PP$-almost surely,
	\begin{equation}\label{A}
		\lim_{n\to\infty} \int_{\R\setminus [-n^{1/4},n^{1/4}]}  R_n(u,x)\, du =0,
	\end{equation}
	and that
	\begin{equation}\label{B}
		\EE\Bigl[ \int_\R 	\liminf_{n\to\infty}\,  \Bigl| 	 \int_{  [-n^{1/4},n^{1/4}]}  R_n(u,x)\, du \Bigr|^2\, dx\Bigr]>0.
	\end{equation}
	Using Fatou's lemma, \eqref{A} and~\eqref{B} we then obtain~\eqref{rsv8} by  
	\begin{equation}\label{rsv9}
		\begin{aligned}
			&	\liminf_{n\to\infty}  \EE\Bigl[ \int_\R  \Bigl| \int_\R  R_n(u,x)\, du\Bigr|^2\, dx\Bigr]\\
			&\qquad\qquad \ge  \EE\Bigl[ \int_\R 	\liminf_{n\to\infty}\,  \Bigl| \int_\R  R_n(u,x)\, du\Bigr|^2\, dx\Bigr]\\
			& \qquad\qquad \ge  \EE\Bigl[ \int_\R 	\liminf_{n\to\infty}\,  \Bigl( \frac{1}{2}\Bigl| \int_{ [-n^{1/4},n^{1/4}]}  R_n(u,x)\, du\Bigr|^2\\
			& \qquad\qquad\qquad\qquad - \Bigl| \int_{\R\setminus [-n^{1/4},n^{1/4}]}  R_n(u,x)\, du\Bigr|^2\Bigr)\, dx\Bigr]\\
			& \qquad\qquad =\frac{1}{2}\, \EE\Bigl[ \int_\R 	\liminf_{n\to\infty}\, \Bigl| \int_{[-n^{1/4}, n^{1/4}]}  R_n(u,x)\, du\Bigr|^2 \, dx\Bigr] >0.
		\end{aligned}
	\end{equation}	
	
	We turn to the proof of~\eqref{A}. For all  $x,u\in\R$, $\omega\in\Omega$ and  $n\in\N$ we obtain by~\eqref{rsv5a} that
	\[
	|R_n(u,x,\omega)| \le c_6 \ln(\sqrt{4n}) \, n^{1/4+s/2} \, e^{-u^2/(2c_4)}.
	\]
	Using a standard estimate for the tail of the standard normal distribution we conclude that for all $x\in\R$ and all $n\in\N$,
	\begin{equation}\label{rsv10}
		\begin{aligned}
			\Bigl|	\int_{\R\setminus  [-n^{1/4},n^{1/4}]}  R_n(u,x,\omega)\, du\Bigr|	& \le c_7 \ln(\sqrt{4n}) \, 
			n^{1/4+s/2} \, \int_{n^{1/4}/\sqrt{c_4}}^\infty e^{-u^2/2}\, du\\
			& \le c_8  \ln(\sqrt{4n})
			\, n^{s/2}\, e^{-\sqrt{n}/(2c_4)},
		\end{aligned}
	\end{equation}			
	which implies~\eqref{A}.
	
	For the proof of~\eqref{B} we first show that 
	\begin{equation}\label{rtl1}
		\EE[|\kappa(1)|^2] > 0.
	\end{equation}
	We have
	\begin{equation}\label{rsv14}
		\begin{aligned}
			\EE\bigl[|\kappa(1)|^2\bigr] & = \EE[\kappa(1)\, \kappa(-1)]\\
			& =  \EE\Bigl[ \int_0^1 \bigl(e^{\bi \, W_t}	- e^{\bi\, \tW_t^{\pi^*} }\bigr)\, dt \, \int_0^1 \bigl(e^{-\bi \, W_t}	-e^{-\bi\, \tW_t^{\pi^*}}\bigr)\, dt \Biggr]\\
			& = 4\int_0^1\int_s^1 \Bigl( \EE\bigl[e^{\bi\, (W_t-W_s)}\bigl] -\EE\bigl[ e^{\bi\, (W_t-\tW^{\pi^*}_s)}\bigr]\Bigr) \, dt \, ds.
		\end{aligned}
	\end{equation}
	Note that for all $0\le s\le t \le 1$ we have $W_t-\tW^{\pi^*}_s = (t-s)W_1 + B^{\pi^*}_t - \tB^{\pi^*}_s$ and $W_1$, $B^{\pi^*}_t$, $ \tB^{\pi^*}_s$ are independent, and therefore,
	\begin{equation}\label{rsv15}
		\EE\bigl[ e^{\bi\, (W_t-\tW^{\pi^*}_s)}\bigr] = \EE\bigl[ e^{\bi\, (t-s)W_1}\bigr]\, \EE\bigl[ e^{\bi\, B^{\pi^*}_t}\bigr]\, \EE\bigl[ e^{\bi\, \tB^{\pi^*}_s}\bigr].
	\end{equation}
	Combining~\eqref{rsv14} with~\eqref{rsv15} we conclude that
	\begin{equation*}\label{rsv16}
		\begin{aligned}
			\EE\bigl[|\kappa(1)|^2\bigr] &  = 4\int_0^1\int_s^1 \Bigl( e^{-(t-s)/2} - e^{-(t-s)^2/2} \, e^{-t(1-t)/2}\, e^{-s(1-s)/2}\Bigr) \, dt \, ds\\
			& = 4\int_0^1\int_s^1 e^{-(t-s)/2}\bigl( 1- e^{-s(1-t)}\bigr)\, dt \, ds,
		\end{aligned}
	\end{equation*}
	which implies~\eqref{rtl1}.

	Clearly, ~\eqref{rtl1} implies $\EE[|\alpha(\kappa(1))|] >0$
	for some $\alpha\in\{\text{Re},\text{Im}\}$ and it follows that there exists $\delta\in (0,\infty)$ such that 
	\begin{equation}\label{rtl2}
		\PP(|\alpha(\kappa(1))| \ge \delta) >0.
	\end{equation}
	Since $\PP(\|W\|_\infty + \|\widetilde W^{\pi^*}\|_\infty<\infty) = 1$ we conclude that there exists $M\in\N$ such that 
	\begin{equation}\label{rtl3}
		\PP\bigl(\{|\alpha( \kappa(1))| \ge \delta\}\cap\{\|W\|_\infty + \|\widetilde W^{\pi^*}\|_\infty\le M\}) >0.
	\end{equation}
	Put
	\[
	A = \{|\alpha( \kappa(1))| \ge \delta\}\cap\{\|W\|_\infty + \|\widetilde W^{\pi^*}\|_\infty\le M\}.
	\]
	By~\eqref{rsv5}
	we derive that for all $\omega\in A$ and all $z,y\in\R$,
	\begin{equation}\label{rtl4}
		|\alpha( \kappa(z,\omega)) - \alpha( \kappa(y,\omega))| = |\alpha(  \kappa(z,\omega) -  \kappa(y,\omega))| \le | \kappa(z,\omega) - \kappa(y,\omega)|\le |z-y|\,M.
	\end{equation}
	Put 
	\[
	\tilde\delta = \min\bigl(\delta/(2M),1/2).
	\]
	Using~\eqref{rtl3} and~\eqref{rtl4} we get $\PP(A)>0$ and for all $\omega\in A$ and all $z\in [1-\tilde\delta,1+\tilde \delta]$,
	\begin{equation}\label{rtl5}
		|\alpha( \kappa(z,\omega))| \ge |\alpha( \kappa(1,\omega))|  - |\alpha(\kappa(z,\omega))-\alpha(\kappa(1,\omega))| \ge \delta- \delta/2 = \delta/2.
	\end{equation}
	Choose $n_0\in\N$ such that $1/n^{1/4} < \tilde \delta $ for all $n\ge n_0$. Then for all $n\ge n_0$,  all 
	$x\in [-1-\tilde\delta/2, -1+\tilde\delta/2]$ 
	and  all $u\in [-n^{1/4},n^{1/4}]$ we have
	\[
	|1-(u/\sqrt{4n} - x)| \le |u|/\sqrt{4n} + |1+x| \le 1/(2n^{1/4}) + \tilde \delta /2 < \tilde \delta.
	\] 
	Employing~\eqref{rtl5} we thus conclude that for all $\omega \in A$, all $n\ge n_0$,  all $x\in [-1-\tilde\delta/2, -1+\tilde\delta/2]$ and all $u\in [-n^{1/4},n^{1/4}]$,
	\begin{equation}\label{rtl6}
		|\alpha( \kappa(u/\sqrt{4n} -x,\omega))| \ge \delta/2.	
	\end{equation}
	Hence, for all $\omega \in A$, all $n\ge n_0$,  all $x\in [-1-\tilde\delta/2, -1+\tilde\delta/2]$ and all $u\in [-n^{1/4},n^{1/4}]$,
	\begin{equation}\label{rtl7}
		\begin{aligned}
			&	|\alpha(R_n(u,x,\omega))| \\
			&\qquad = \frac{ e^{-u^2/(2c_4)}\ln(\sqrt{4n}) }{(e/\sqrt{4n} + |u/\sqrt{4n}-x|)^{1/2+s}\, \ln(e + |u-\sqrt{4n}x|)}|\alpha(\kappa(u/\sqrt{4n}-x,\omega))|\\
			&\qquad \ge \frac{ e^{-u^2/(2c_4)}\ln(\sqrt{4n}) }{(e/\sqrt{4n} + |u/\sqrt{4n}-x|)^{1/2+s}\, \ln(e + |u-\sqrt{4n}x|)}\,\delta/2.
		\end{aligned}
	\end{equation}
	Moreover,
	for all $\omega \in A$,   
	all $x\in [-1-\tilde\delta/2, -1+\tilde\delta/2]$ 
	and all $u\in \R$,
	\begin{equation}\label{rtl7a}
		\liminf_{n\to\infty} \frac{ e^{-u^2/(2c_4)}\ln(\sqrt{4n}) }{(e/\sqrt{4n} + |u/\sqrt{4n}-x|)^{1/2+s}\, \ln(e + |u-\sqrt{4n}x|)} = \frac{ e^{-u^2/(2c_4)} }{ |x|^{1/2+s}} \ge  \frac{e^{-u^2/(2c_4)} }{(3/2)^{1/2+s}}.
	\end{equation}
	Using~\eqref{rtl7} and \eqref{rtl7a} and Fatou's lemma we conclude that
	\begin{equation*}\label{rtl8}
		\begin{aligned}
			& \EE\Bigl[ \int_\R 	\liminf_{n\to\infty}\, \Bigl| \int_{-n^{1/4}}^{ n^{1/4}}  R_n(u,x)\, du\Bigr|^2 \, dx\Bigr]	\\
			& \quad \ge \int_A\int_{-1-\tilde\delta/2}^{-1+\tilde\delta/2}	\liminf_{n\to\infty} \, \Bigl| \int_{-n^{1/4}}^{n^{1/4}}   R_n(u,x, \omega)\, du\Bigr|^2 \, dx\, \PP(d\omega)	\\
			& \quad\ge \int_A \int_{-1-\tilde\delta/2}^{-1+\tilde\delta/2}	\liminf_{n\to\infty} \, \Bigl| \int_{-n^{1/4}}^{n^{1/4}}  \alpha(R_n(u,x, \omega))\, du\Bigr|^2 \, dx\, \PP(d\omega)	\\
			& \quad\ge 
			\int_A \int_{-1-\tilde\delta/2}^{-1+\tilde\delta/2}	\liminf_{n\to\infty} \, \Bigl| \int_{-n^{1/4}}^{n^{1/4}}  \frac{ e^{-u^2/(2c_4)}\ln(\sqrt{4n}) \,\delta/2}{(e/\sqrt{4n} + |u/\sqrt{4n}-x|)^{1/2+s}\, \ln(e + |u-\sqrt{4n}x|)} \, du\Bigr|^2 \, dx\, \PP(d\omega)\\
			& \quad\ge \int_A \int_{-1-\tilde\delta/2}^{-1+\tilde\delta/2} \, \Bigl| \int_\R 	\liminf_{n\to\infty} \frac{ 1_{[-n^{1/4},n^{1/4}]}(u) \,e^{-u^2/(2c_4)}\ln(\sqrt{4n}) \,\delta/2}{(e/\sqrt{4n} + |u/\sqrt{4n}-x|)^{1/2+s}\, \ln(e + |u-\sqrt{4n}x|)} \, du\Bigr|^2 \, dx\, \PP(d\omega)\\
			& \quad\ge c_9\,\int_A \int_{-1-\tilde\delta/2}^{-1+\tilde\delta/2} \, \Bigl| \int_\R  e^{-u^2/(2c_4)}\, du\Bigr|^2\, dx\, \PP(d\omega) \\
			& \quad= c_9\, \Bigl| \int_\R  e^{-u^2/(2c_4)}\, du\Bigr|^2\, \tilde \delta \, \PP(A) > 0.
		\end{aligned}
	\end{equation*}	
	This completes the proof of the lemma.	
\end{proof}

We are ready to prove Theorem~\ref{thm2}. Let $s\in (1/2,1)$ and let $X$ denote the strong solution of the SDE~\eqref{sde0} with $\mu = \mu_s$. By Lemma~\ref{lem3}, the function $\mu_s$ is bounded and satisfies $\mu_s\in \Lc^1(\R)$ as well as $\mu_s\in \Wc^{s,2}\cap \Wc^{(s-1/4)-,4}$. In particular, $\mu_s$ satisfies all assumptions in Theorem~\ref{thm2} and Lemmas~\ref{lemf1} to~\ref{lemf3} are applicable with $\mu = \mu_s$.

Recall the definition of $\tX^\pi$ for $\pi\in\Pi$, see~\eqref{extra1}, and put $Y=G_{\mu_s}\circ X$ and $\tY^\pi=G_{\mu_s}\circ \tX^\pi$, see~\eqref{extra2} and ~\eqref{extra3}. 
Using~\eqref{ndisc2} as well as Lemma~\ref{lemf1} and Lemma~\ref{lemf2} we obtain that there exists $c\in (0,\infty)$ such that for all $n\in\N$,
\begin{equation}\label{end1}
	\begin{aligned}
		\inf_{\pi\in\Pi^n} e_2(\pi) & \ge \inf_{\pi\in\widetilde \Pi^n} e_2(\pi) \ge \frac{1}{2} \inf_{\pi\in\widetilde \Pi^n} \EE\bigl[|X_1-\tX^{\pi}_1|^2\bigr]^{1/2} \ge c \inf_{\pi\in\widetilde \Pi^n} \EE\bigl[|Y_1-\tY^{\pi}_1|^2\bigr]^{1/2}.
	\end{aligned}	
\end{equation} 
Combining Lemma~\ref{lemf3b} with Lemma~\ref{lemf5} and observing~\eqref{ndisc3} we get the existence of $n^*\in \N$ and $c_0,c_1\in (0,\infty)$ such that for all $n\ge n^*$ and all $\pi =\{t_1,\dots,t_{5n}\}\in \widetilde \Pi^n$ with $0 < t_1 <\dots t_{5n}=1$,
\begin{equation}\label{end2}
	\begin{aligned}
		\EE\bigl[|Y_1-\tY^{\pi}_1|^2\bigr] & \ge c_1 \frac{\#\bigl\{ i\in \{2,\dots, 5n\}\colon t_{i-1}\ge 1/2\text{ and } t_i-t_{i-1} = 1/(4n)\} }{\ln^2(n+1)n^{2+s}}  - \frac{1}{n^{1+s+c_0}}\\
		& \ge \frac{c_1}{\ln^2(n+1)n^{1+s}}  - \frac{1}{n^{1+s+c_0}}.
	\end{aligned}	
\end{equation} 
Combining~\eqref{end1} with \eqref{end2} we obtain that there exists $n_0 \in \N$ such that for all $n\geq n_0$, 
\[
\inf_{\pi\in\Pi^n} e_2(\pi)  \ge  \frac{c \sqrt{c_1}}{2\ln(n+1)n^{(1+s)/2}}. 
\]
Using the fact that the sequence $(\inf_{\pi\in\Pi^n} e_2(\pi))_{n\in\N}$ is monotonically decreasing we conclude that there exists $c_2\in (0,\infty)$ such that for all $n\in\N$,
\begin{align*}
	\inf_{\substack{
			t_1,\dots ,t_n \in [0,1]\\
			g \colon \R^n \to \R \text{ measurable} \\
	}}	 \EE\bigl[|X_1-g(W_{t_1}, \ldots, W_{t_n})|^2\bigr]^{1/2}\geq  \inf_{\pi\in\Pi^{n+1}} e_2(\pi)
	\geq \frac{c_2 }{\ln(n+1)n^{(1+s)/2}},
\end{align*} 			
which completes the proof of Theorem~\ref{thm2}.

\bibliographystyle{acm}
\bibliography{bibfile}

\end{document}